\documentclass[11pt]{amsart}

\usepackage{geometry}
\usepackage{todonotes}

\usepackage{amsfonts, amssymb, amscd}
\numberwithin{equation}{section}

\usepackage[symbol]{footmisc}

\usepackage{bm}
\usepackage{verbatim}
\usepackage{mathrsfs}
\usepackage{graphicx}
\usepackage{tikz-cd}
\usepackage{subcaption}
\usepackage{listings}
\usepackage{subfiles}
\usepackage[toc,page]{appendix}
\usepackage{mathtools}
\usepackage{comment}
\usepackage{enumerate}
\usepackage{enumitem}
\usepackage[all]{xy}

\usepackage{graphicx}
\graphicspath{{images/}}

\usepackage{appendix}
\usepackage{hyperref}
\hypersetup{
    colorlinks=true,
    citecolor=red,
    linkcolor=blue,
    filecolor=magenta,      
    urlcolor=red,
}
\newcommand{\bb}{\bm{b}}
\newcommand{\Mm}{{\bf{M}}}
\newcommand{\PP}{{\bf{P}}}
\newcommand{\NN}{{\bf{N}}}
\newcommand{\Dd}{{\bf{D}}}

\newcommand{\Qq}{\mathbb{Q}}

\newcommand{\Rr}{\mathbb{R}}

\newcommand{\Center}{\operatorname{center}}

\newcommand{\Exc}{\operatorname{Exc}}

\newcommand{\Ima}{\operatorname{Im}}
\newcommand{\Nklt}{\operatorname{Nklt}}

\newcommand{\Supp}{\operatorname{Supp}}
\newcommand{\Ngklt}{\operatorname{Ngklt}}
\newcommand{\Nlc}{\operatorname{Nlc}}

\newcommand{\mult}{\operatorname{mult}}

\newcommand{\Div}{\operatorname{Div}}
\newcommand{\cont}{\operatorname{cont}}

\newcommand{\Ii}{\mathcal{I}}

\newcommand{\Pic}{\mathrm{Pic}}

\newtheorem{thm}{Theorem}[section]

\newtheorem{cor}[thm]{Corollary}
\newtheorem{lem}[thm]{Lemma}

\newtheorem{claim}[thm]{Claim}

\theoremstyle{definition}
\newtheorem{defn}[thm]{Definition}

\theoremstyle{definition}
\newtheorem{rem}[thm]{Remark}

\newtheorem{deflem}[thm]{Definition-Lemma}

\newtheorem{nota}[thm]{Notation}

\theoremstyle{definition}

\begin{document}

\title{Existence of flips for generalized lc pairs}
\author{Christopher D. Hacon and Jihao Liu}

\address{Department of Mathematics, The University of Utah, Salt Lake City, UT 84112, USA}
\email{hacon@math.utah.edu}

\address{Department of Mathematics, The University of Utah, Salt Lake City, UT 84112, USA}
\email{jliu@math.utah.edu}

\date{\today}

\begin{abstract}
We prove the existence of flips for $\Qq$-factorial NQC generalized lc pairs, and the cone and contraction theorems for NQC generalized lc pairs. This answers a question of C. Birkar which was conjectured by J. Han and Z. Li. As an immediate application, we show that we can run the minimal model program for $\Qq$-factorial NQC generalized lc pairs. In particular, we complete the minimal model program for $\Qq$-factorial NQC generalized lc pairs in dimension $\leq 3$ and pseudo-effective $\Qq$-factorial NQC generalized lc pairs in dimension $4$.
\end{abstract}

\maketitle

\tableofcontents

\section{Introduction}

We work over the field of complex numbers $\mathbb C$, however many of the results also hold over any algebraically closed field $k$ of characteristic zero.

The main purpose of this paper is to prove the following theorem:

\begin{thm}\label{thm: existence of glc closure}
Let $(X,B,\Mm)/U$ be an NQC glc g-pair and $U^0\subset U$ a non-empty open subset. Let  $X^0:=X\times_UU^0$, $B^0:=B\times_UU^0$, and $\Mm^0:=\Mm\times_UU^0$ (see Definition \ref{defn: b divisor restriction over an open subset} below). Assume that
\begin{enumerate}
\item the morphism $X\rightarrow U$ is a projective morphism between normal quasi-projective varieties,
\item $(X^0,B^0,\Mm^0)/U^0$ has a good minimal model,
\item all glc centers of $(X,B,\Mm)$ intersect $X^0$, and
\item $\Mm ^0$ descends to $X^0$ and $\Mm^0_{X^0}\sim_{\mathbb R,U^0}0$.
\end{enumerate}
Then $(X,B,\Mm)/U$ has a good minimal model.
\end{thm}

Here ``glc g-pair" stands for ``generalized lc pair" and ``NQC" stands for ``nef $\Qq$-Cartier combination". We remark that NQC generalized pairs in our paper exactly correspond to the original generalized pairs defined in \cite{BZ16}. See Definition \ref{defn: g-pairs} for more details.

As an immediate application of Theorem \ref{thm: existence of glc closure}, we show the existence of flips for $\Qq$-factorial NQC generalized lc pairs:

\begin{thm}[Existence of generalized lc flips]\label{thm: existence of q-factorial glc flips}
Let $(X,B,\Mm)/U$ be a $\Qq$-factorial NQC glc g-pair and $f: X\rightarrow Z$ a $(K_X+B+\Mm_X)$-flipping contraction over $U$. Then the flip $f^+: X^+\rightarrow Z$ of $f$ exists.
\end{thm}
In fact, we will also show that $X^+$ is $\Qq$-factorial and $\rho(X)=\rho(X^+)$, and hence the flip $X\dasharrow X^+$ is compatible with the minimal model program for generalized pairs. See Theorem \ref{thm: existence glc flip with m r cartier} below.

\medskip

As a complement to Theorem \ref{thm: existence of q-factorial glc flips}, we prove the cone and contraction theorems for NQC generalized lc pairs, thus completely answering Birkar's question on the existence of contractions and flips \cite[6.1]{Bir20b} which was originally conjectured by Han-Li \cite[Conjectures 3.1, 3.3]{HL18}.

\begin{thm}[Cone and contraction theorems for generalized lc pairs]\label{thm: cone and contraction theorem glc pair}
Let $(X,B,\Mm)/U$ be an NQC glc g-pair and $\pi: X\rightarrow U$ the associated morphism. Let $\{R_j\}_{j\in\Lambda}$ be the set of $(K_X+B+\Mm_X)$-negative extremal rays in $\overline{NE}(X/U)$ that are rational. Then:
\begin{enumerate}
    \item $$\overline{NE}(X/U)=\overline{NE}(X/U)_{K_X+B+\Mm_X\geq 0}+\sum_{j\in\Lambda} R_j.$$
    In particular, any $(K_X+B+\Mm_X)$-negative extremal ray in $\overline{NE}(X/U)$ is rational.
    \item Each $R_j$ is spanned by a rational curve $C_j$ such that $\pi(C_j)=\{pt\}$ and 
    $$0<-(K_X+B+\Mm_X)\cdot C_j\leq 2\dim X.$$
    \item For any ample$/U$ $\Rr$-divisor $A$ on $X$,
    $$\Lambda_A:=\{j\in\Lambda\mid R_j\subset\overline{NE}(X/U)_{K_X+B+\Mm_X+A<0}\}$$
    is a finite set. In particular, $\{R_j\}_{j\in\Lambda}$ is countable, and is a discrete subset in $\overline{NE}(X/U)_{K_X+B+\Mm_X+A<0}$. Moreover, we may write
    $$\overline{NE}(X/U)=\overline{NE}(X/U)_{K_X+B+\Mm_X+A\geq 0}+\sum_{j\in\Lambda_A}R_j.$$
    \item Assume that $\Mm_X$ is $\Rr$-Cartier. Let $R$ be a $(K_X+B+\Mm_X)$-negative extremal ray in $\overline{NE}(X/U)$. Then $R$ is a rational extremal ray. In particular, there exists a projective morphism $\cont_R: X\rightarrow Y$ over $U$ satisfying the following.
    \begin{enumerate}
        \item For any integral curve $C$ such that $\pi(C)$ is a point, $\cont_R(C)$ is a point if and only if $[C]\in R$.
        \item $\mathcal{O}_Y\cong(\cont_R)_*\mathcal{O}_X$. In other words, $\cont_R$ is a contraction.
        \item Let $L$ be a line bundle on $X$ such that $L\cdot R=0$. Then there exists a line bundle $L_Y$ on $Y$ such that $L\cong f^*L_Y$.
    \end{enumerate}
\end{enumerate}
\end{thm}

An immediate corollary of Theorems \ref{thm: existence of q-factorial glc flips} and \ref{thm: cone and contraction theorem glc pair} is that we can run minimal model programs (MMPs) for $\Qq$-factorial NQC glc g-pairs:

\begin{thm}\label{thm: can run gpair mmp}
We can run the MMP for $\Qq$-factorial NQC glc g-pairs. More precisely, for any $\Qq$-factorial NQC glc g-pair $(X,B,\Mm)/U$, there exists a sequence of $(K_X+B+\Mm_X)$-flips and divisorial contractions over $U$. Moreover, any such sequence ends either with a Mori fiber space, or a minimal model, or an infinite sequence of flips over $U$.
\end{thm}

Therefore, as long as we know the termination of flips, we can completely establish the minimal model program for $\Qq$-factorial NQC glc g-pairs. In particular, we have:

\begin{thm}\label{thm: gpair mmp 3fold and pe fourfold}
The MMP for $\Qq$-factorial NQC glc g-pairs in dimension $\leq 3$ holds, and the MMP for pseudo-effective $\Qq$-factorial NQC glc g-pairs in dimension $4$ holds. More precisely, for any $\Qq$-factorial NQC glc g-pair $(X,B,\Mm)/U$ such that either $\dim X\leq 3$ or $\dim X=4$ and $K_X+B+\Mm_X$ is pseudo-effective$/U$, there exists a sequence of $(K_X+B+\Mm_X)$-flips and divisorial contractions over $U$. Moreover, any such sequence ends either with a Mori fiber space or a minimal model over $U$.

\end{thm}

\medskip

The theory of \emph{generalized pairs} (\emph{g-pairs} for short) was introduced by C. Birkar and D.-Q. Zhang in \cite{BZ16} to tackle the effective Iitaka fibration conjecture. Some embryonic forms of this theory can be found in \cite{Bir12b,BH14}, and can even be traced back to the early studies on the moduli part of the canonical bundle formula and sub-adjunctions \cite{Kaw98,FM00}. Although, seemingly technical, in recent years, the theory of generalized pairs has proven to be a powerful tool in birational geometry. In particular, this theory has been essentially used in the proof of the Borisov-Alexeev-Borisov conjecture \cite{Bir19,Bir21a}. For other results closely related to the theory of generalized pairs, we refer the reader to \cite{HX15,Fil18a,Mor18,HL18, Fil18b,Bir18,HH19,HL19,LT19,HM20,HL20a,HL20b,HL20d,LP20a,LP20b,Li20,Hu20,FS20a,Fil20,Bir20a,HL20c,CX20,Bir20c,FS20b,FW20,BDCS20,CT20,Sho20,Has20,Li21,Liu21,LX21,Hu21,Jia21,Bir21b,FH21}. We also refer the reader to \cite{Bir20b} for a more detailed introduction to the theory of generalized pairs.

It has recently become apparent that the minimal model program (MMP) for generalized pairs is closely related to the minimal model program for usual pairs and varieties. In particular, generalized pairs have been used to prove the termination of pseudo-effective fourfold flips \cite{Mor18,HL18,HM20,CT20}. For this, and other reasons, it is important to study the minimal model program for generalized pairs. For gklt (generalized klt) g-pairs or $\Qq$-factorial gdlt (generalized dlt) g-pairs, the corresponding theory is very similar to the case of the usual klt or $\Qq$-factorial dlt pairs (cf. \cite[Lemma 4.4]{BZ16}, \cite[Lemma 3.5]{HL18}). However, when studying the MMP for glc g-pairs, we encounter several non-trivial issues. Before discussing these, let us first recall the main features of the usual minimal model program. 

\medskip

\noindent\textbf{Step 1}. We start with a $\Qq$-factorial projective pair $(X,B)$ with at worst lc singularities.

\medskip

\noindent\textbf{Step 2}. If $K_X+B$ is nef, i.e. $(X,B)$ is a minimal model then we are done. Otherwise, by the cone and the contraction theorems, we contract a $(K_X+B)$-negative extremal ray and get a contraction $f: X\rightarrow Z$.

\medskip

\noindent\textbf{Step 3}. If $f$ is a Mori fiber space, then we are done. If $f$ is a divisorial contraction, then we replace $X$ with $Z$ and continue. If $f$ is a flipping contraction and the flip $f^+: X^+\rightarrow Z$ of $f$ exists, then we replace $X$ with $X^+$ and continue. Note that the $\Qq$-factorial condition, which usually follows from the cone and the contraction theorems, needs to be preserved.

\medskip

\noindent\textbf{Step 4}. If there does not exist an infinite sequence of flips, then the MMP terminates with either a minimal model or a Mori fiber space, and we are done.

\medskip

In summary, to complete a minimal model program, we need
\begin{enumerate}
    \item the cone and the contraction theorems,
    \item the existence of flips, and
    \item the termination of flips.
\end{enumerate}

For the usual lc pairs, (1) and (2) are completely known. In fact, the cone and contraction theorems for projective klt pairs appear in \cite{Kaw84} and are completed in \cite{Kol84}, the relative versions are proven in \cite{KMM87}, the cone and contraction theorems for lc pairs are proven in \cite{Amb03,Fuj11} by using the theory of quasi-log varieties, the existence of klt flips is proven in \cite{BCHM10}, and the existence of lc flips is proven in \cite{Bir12a,HX13}. The difficult part for the minimal model program for usual pairs is (3): we only know the termination of flips in full generality, in dimension $\leq 3$ \cite{Kaw92,Sho96}. In dimension $4$, some special cases have been proven: the terminal case in \cite{KMM87}, klt anti-effective case and some other special cases in \cite{AHK07}, canonical case with rational coefficients in \cite{Fuj04,Fuj05}, effective case in \cite{Bir07,HMX14}, and pseudo-effective case  \cite{Mor18,HL18,HM20,CT20} as we mentioned earlier.

For glc pairs that are neither gklt nor $\Qq$-factorial gdlt, the situation is completely different. First of all, we usually need to add the NQC condition for technical reasons (cf. \cite[Example 3.15]{HL18}), however this is a natural assumption and is contained in the original definition of generalized pairs in \cite{BZ16}. Under the NQC assumption, the known results on the termination of flips are similar to the usual pair case (in particular, in full generality in dimension $\leq 3$ \cite{CT20} and in the pseudo-effective case in dimension $4$ \cite{HM20,CT20}). However, the cone and contraction theorems and the existence of flips seem to be far more challenging even in dimension $3$, and we only know some partial results when $\Mm$ descends to $X$, i.e., $\Mm_X$ is nef \cite{LP20a,LP20b}. We remark that for curves and surfaces, there are no flips, and the cone and contraction theorems follow from the usual cone and contraction theorems as $\Mm_X$ is always nef. 

Very recently, there has been some progress towards the existence of flips for generalized pairs. K. Hashizume has shown the non-vanishing theorem for glc pairs with a polarization  \cite[Theorem 1.1]{Has20}, and also investigated the termination of a special MMP for glc pairs \cite[Theorem 1.3]{Has20}, which has a close connection with the existence of glc flips. Recently, Z. Hu proved the finiteness of B-representations for some special glc pairs \cite[Theorem 1.5]{Hu21}, a generalized pair version of \cite[Theorem 1.1]{FG14} and \cite[Theorem 1.2]{HX16}. This result induces the generalized pair version of \cite[Theorem 1.4]{FG14} and \cite[Theorem 1.4]{HX16}  (\cite[Thereom 1.9]{Hu21}), and \cite[Theorem 1.7]{Bir12a} and \cite[Corollary 1.5]{HX16}  (\cite[Thereom 1.10]{Hu21}), which are important theorems in the proof of the existence of lc flips. In this paper, we will prove the existence of flips for $\Qq$-factorial NQC glc g-pairs, which solves (2). Although it turns out that we do not need  K. Hashizume's result and Z. Hu's result to prove the existence of glc flips, their results motivated us. We also expect that the methods in \cite{Hu21} can be applied to the termination of flips under the setting of \cite[Theorem 1.3]{Has20}. 

To guarantee that we can run the minimal model program for $\Qq$-factorial glc g-pairs, we are only left to prove the corresponding cone and contraction theorems. We can prove the cone theorem by combining results of F. Ambro and O. Fujino on the cone and contraction theorems of non-lc pairs \cite{Amb03,Fuj11} and recent results of J. Han and W. Liu  on generalized sub-adjunction \cite{HL19}. Although we cannot prove the base-point-free theorem in the same generality as in \cite{Amb03,Fuj11}, fortunately for us, we are able to prove the contraction theorem for extremal rays by using the special properties of extremal rays. Combining with the cone theorem, we solve (1). We remark that O. Fujino has a recent paper \cite{Fuj21} on related topics, where he considers the cone and contraction theorems for quasi-log schemes. Note that any quasi-lc pair is a generalized lc pair (cf. \cite[Remark 1.9]{Fuj18}).

\medskip

\noindent\textit{Sketch of the proof}. We first sketch our proof on the existence of flips (Theorem \ref{thm: existence of q-factorial glc flips}). Notice that for a $\Qq$-factorial generalized pair $(X,B,\Mm)/U$ equipped with a flipping contraction $f: X\rightarrow Z$ over $U$, we can always assume that $\Mm_X\cdot C<0$ for any flipping curve, otherwise $f$ is a $(K_X+B)$-flipping contraction and the existence of the flip just follows from the existence of flips for lc pairs. Notice that on the complement of the flipping locus, we have $\Mm_X\sim_{\Rr}0$ as a $\bb$-divisor over $Z$. In other words, $\Mm_X|_{f^{-1}(Z^0)}\sim_{\Rr,Z^0}0$ over some non-empty open subset $Z^0$ of $Z$. However, in this situation, we can pick $0\leq G\sim_{\Rr,Z}\Mm_X$ such that $(X,B+G)$ is lc and have some additional good properties. The existence a good minimal model of $(X,B,\Mm)/U$ follows from the existence of a good minimal model of $(X,B+G)$ which is preserved by \cite{HX13,Has19}.

Now we sketch a proof of the cone theorem. Possibly perturbing the generalized pair $(X,B,\Mm)/U$ with an ample divisor, we may assume that $K_X+B+\Mm_X\sim_{\Rr,U}K_X+\Delta$ for some pair $(X,\Delta)$. $(X,\Delta)$ is not necessary lc; however, we may assume that the non-lc locus of $(X,\Delta)$ is exactly the non-gklt locus of $(X,B,\Mm)$. Now we want to apply the results of F. Ambro \cite{Amb03} and O. Fujino \cite{Fuj11}. Since we have a generalized pair structure, a key observation is that we can do sub-adjunction to any non-lc center of $(X,B,\Mm)$ and still have a generalized lc pair structure after the sub-adjunction \cite{HL19}. Applying induction on the dimension, we immediately get the cone theorem. 

We can use the cone theorem to deduce the contraction theorem. Assume that $\Mm_X$ is $\Rr$-Cartier. To prove the contraction theorem for any $(K_X+B+\Mm_X)$-negative extremal ray $R$, we only need to prove some special base-point-freeness theorems for any supporting function $L$ of a $(K_X+B+\Mm_X)$-negative extremal ray. We can always assume that $\Mm_X\cdot R<0$. In this case, possibly replacing $\Mm$ with $(1-\epsilon)\Mm$ for $0<\epsilon\ll 1$, we may assume that $\Ngklt(X,B,\Mm)=\Nklt(X,B)=\Nlc(X,\Delta)$ and $\Nlc(X,\Delta)$ does not contain any curve $C$ such that $[C]\in R$. By the cone theorem and Kleiman's Criterion, we know that $L|_{\Nlc(X,\Delta)}$ is ample. Now we can apply the results of F. Ambro \cite{Amb03} and O. Fujino \cite{Fuj11} to prove the base-point-freeness theorem for $L$. 

Finally, together with the existence of flips we just proved, we know that we can run MMP for any $\Qq$-factorial NQC glc g-pair (Theorem \ref{thm: can run gpair mmp}). In other words, for $\Qq$-factorial NQC glc g-pairs, whenever we know the termination of flips, we will have the complete MMP, and Theorem \ref{thm: gpair mmp 3fold and pe fourfold} follows from \cite{HM20,CT20}.

\medskip

\noindent\textit{Structure of the paper}. In Section 2, we introduce some notation and tools for generalized pairs and MMPs which will be used in this paper. In Section 3, we study different models for generalized pairs. In Section 4, we proof prove Theorem \ref{thm: existence of glc closure}. In Section 5, we prove the base-point-free theorem, contraction theorem, and cone theorem for generalized lc pairs, which implies Theorem \ref{thm: cone and contraction theorem glc pair}. In Section 6, we prove the rest of our main theorems, i.e. Theorems \ref{thm: existence of q-factorial glc flips}, \ref{thm: can run gpair mmp}, and \ref{thm: gpair mmp 3fold and pe fourfold}.

\medskip

\noindent\textbf{Postscript}. After the first version of our paper appeared, we were informed by Haidong Liu and Zhengyu Hu that there are some troubles with \cite[Theorem 1.9]{Hu21}, a result we used in the proofs of Theorems 7.5 and 9.1 of the first version of the paper. Therefore some of the corresponding  proofs in the first version of this paper are incomplete. Fortunately for us, we could avoid using \cite{Hu21} and still prove all our main results except Theorem \ref{thm: cone and contraction theorem glc pair}(4): without applying \cite{Hu21}, we could only contract $(K_X+B+\Mm_X)$-negative extremal rays rather than faces, and we need to assume that $\Mm_X$ is $\Rr$-Cartier. Nevertheless, this modification does not affect the other main results of the paper. Indeed, the current Theorem \ref{thm: cone and contraction theorem glc pair}(4.c) is even comparably stronger than that of the first version from the perspective of running the minimal model program: for example, the current Theorem \ref{thm: cone and contraction theorem glc pair}(4.c) will imply the $\Qq$-factorial glc version of Lemma \ref{lem: still an mmp under perturbation} while the first version did not. We were later informed by Nikolaos Tsakanikas that he and Vladimir Lazi\'c have obtained some results on the minimal model program for generalized pairs \cite{LT21} based on the main theorems of our paper, which require the current version of Theorem \ref{thm: cone and contraction theorem glc pair}(4.c).

After the second version of this paper  appeared,  we  were  informed  by  Kenta Hashizume that Theorem 1.1 can be immediately implied by Claim 7.6 of the second version of our paper on the arXiv.  This greatly simplifies our proof.  Comparing to the second version, we completely remove Sections 5,6, most parts of Sections 4,7,8, and a part of Section 2.  The main theorems are not changed.

We thank Kenta Hashizume, Zhengyu Hu, Haidong Liu, and Nikolaos Tsakanikas for these discussions.

\medskip

\noindent\textbf{Acknowledgement}.  The second author would like to thank Jingjun Han, Junpeng Jiao, and Yuchen Liu for useful discussions. The authors are partially supported by NSF research grants no: DMS-1801851, DMS-1952522 and by a grant from the Simons Foundation; Award Number: 256202.

\section{Preliminaries}

We will freely use the notation and definitions from \cite{KM98,BCHM10}. For generalized pairs, we will follow the definitions in \cite{HL18} but follow the notation as in \cite{FS20b,Has20} (see Remarks \ref{rem: difference in definition of g-pair}, \ref{rem: difference in definition of gdlt}, \ref{rem: difference in notation of g-pair} below).

\subsection{Divisors}
\begin{defn}
Let $a$ be a real number, $X$ a normal variety, and $D=\sum_i d_iD_i$ an $\Rr$-divisor on $X$, where $D_i$ are the irreducible components of $D$. We define $D^{\leq a}:=\sum_{i\mid d_i\leq a} d_iD_i$, $D^{=a}:=\sum_{i\mid d_i=a} d_iD_i$,  $D^{\geq a}:=\sum_{i\mid d_i\leq a} d_iD_i$, $\lfloor D\rfloor:=\sum_i\lfloor d_i\rfloor D_i$, and $\{D\}:=\sum_i\{d_i\}D_i$.
\end{defn}


\begin{defn}
Let $\phi: X\dashrightarrow Y$ be a birational map. We let $\Exc(\phi)$ be the union of the exceptional divisors of $\phi$, and usually identify $\Exc(\phi)$ with the reduced exceptional divisor of $\phi$.
\end{defn}

The following lemma and its proof are taken verbatim from \cite[Lemma 3.2.1]{BCHM10}. 
\begin{lem}[cf. {\cite[Lemma 3.2.1]{BCHM10}}]\label{lem: lift equivalence from generic fiber}
Let $\mathbb K=\mathbb Q$ or $\mathbb R$. Let $\pi: X\rightarrow U$ be a projective morphism between normal quasi-projective varieties. Let $D$ be a $\mathbb K$-Cartier $\mathbb K$-divisor on $X$ and let $D'$ be its restriction to the generic fiber of $\pi$.

If $D'\sim_{\mathbb K} B'\geq 0$ for some $\mathbb K$-divisor $B'$ on the generic fiber of $\pi$, then $D\sim_{\mathbb K,U}B\geq 0$ for some $\mathbb K$-divisor $B$, such that $B'$ is the restriction of $B$ to the generic fiber of $\pi$. 
\end{lem}
\begin{proof}
Taking the closure of the generic points of $B'$, we may assume that there exists a $\mathbb K$-divisor $B_1\geq 0$ such that $B'$ is the restriction of $B_1$ to the generic fiber of $\pi$. Since $D'-B'\sim_{\mathbb K} 0$, $(D-B_1)|_{\pi^{-1}(U_1)}\sim_{\mathbb K} 0$ for some non-empty proper open subset $U_1$ of $U$. Then there exists a $\mathbb K$-divisor $G$ on $X$ such that $D-B_1\sim_{\mathbb K} G$ and $Z:=\pi(\Supp G)$ is a proper closed subset of $U$. Since $U$ is quasi-projective, there exists an ample $\mathbb K$-divisor $H\geq 0$ on $U$ which contains $Z$, such that $F:=\pi^*H\geq -G$. Thus $D\sim_{\mathbb K,U} B_1+F+G\geq 0$. By our construction, $F$ and $G$ are vertical over $U$, so $B'$ is the restriction of $B:=B_1+F+G$ to the generic fiber of $\pi$. 
\end{proof}

\subsection{Rational maps}

\begin{defn}[Contraction and birational contraction]
A \emph{contraction} is a projective morphism $f: X\rightarrow Y$ such that $f_*\mathcal{O}_X=\mathcal{O}_Y$. In particular, $f$ has connected fibers, and if $X\rightarrow Z\rightarrow Y$ is the Stein factorization of $f$, then $Z\rightarrow Y$ is an isomorphism. Moreover, if $X$ is normal, then $Y$ is normal.

Let $\phi: X\dashrightarrow Y$ be a proper birational map between normal varieties. Then $\phi$ is called a \emph{birational contraction} if $\phi$ does not extract any divisors. 
\end{defn}

\begin{defn}
For any birational contraction $\phi: X\dashrightarrow Y$ and $\Rr$-Cartier $\Rr$-divisor $D$ on $X$, let $p: W\rightarrow X$ and $q: W\rightarrow Y$ be a common resolution such that $q=\phi\circ p$, and let $D_Y:=\phi_*D$. Then $f$ is called
\begin{enumerate}
    \item \emph{$D$-trivial} if $D_Y$ is $\Rr$-Cartier and $p^*D=q^*D_Y$,
    \item \emph{$D$-non-positive} if $D_Y$ is $\Rr$-Cartier, and $p^*D=q^*D_Y+E$ for some $E\geq 0$ that is exceptional over $Y$, and
    \item \emph{$D$-negative} if $D_Y$ is $\Rr$-Cartier, $p^*D=q^*D_Y+E$ for some $E\geq 0$ that is exceptional over $Y$, and $\Exc(\phi)\subset\Supp (p_*E)$.
\end{enumerate}
\end{defn}

\begin{lem}\label{lem: trivial map keep property over an open subset}
Let $\pi: X\rightarrow U$ be a projective morphism such that $X$ is normal. Let $D$ be an $\Rr$-Cartier $\Rr$-divisor on $X$, $\phi: X\dashrightarrow Y$ a birational contraction over $U$ such that $\phi$ is $D$-trivial, and $D_Y:=\phi_*D$. Then for any non-empty open subset $U^0\subset U$, $D^0:=D\times_UU^0$ is semi-ample$/U^0$ if and only if $D_Y^0:=D_Y\times_UU^0$ is  semi-ample$/U^0$.
\end{lem}
\begin{proof}
Let $p: W\rightarrow X$ and $q: W\rightarrow Y$ be a common resolution such that $q=\phi\circ p$. Since $\phi$ is $D$-trivial, $D_W:=p^*D=q^*D_Y$. Let $D_W^0:=D_W\times_UU^0$ and $W^0:=W\times_UU^0$. Then
$$(p|_{W^0})^*D^0=(p^*D)|_{W^0}=D_W^0=(q^*D_Y)|_{W^0}=(q|_{W^0})^*D^0_Y,$$
and the lemma follows.
\end{proof}

\subsection{Semi-stable reduction}

\begin{defn}
A pair $(X,B)$ is called \emph{quasi-smooth} if $X$ is $\Qq$-factorial and $(X,B)$ is toroidal.
\end{defn}


\begin{thm}\label{thm: has19 weak semistable reduction}
Let $(X,B)$ be a dlt pair and $\pi: X\rightarrow U$ a projective surjective morphism over a normal variety $U$. Then there exists a commutative diagram of projective morphisms
\begin{center}$\xymatrix{
Y\ar@{->}[r]^{f}\ar@{->}[d]_{\pi'} & X\ar@{->}[d]^{\pi}\\
V\ar@{->}[r]^{\varphi} & U
}$
\end{center}
such that
\begin{enumerate}
    \item $f,\varphi$ are birational morphisms, $\pi'$ is an equidimensional contraction, $Y$ only has $\Qq$-factorial toroidal singularities, and $V$ is smooth, and
    \item there exist two $\Rr$-divisors $B_Y$ and $E$ on $Y$, such that
    \begin{enumerate}
    \item $K_Y+B_Y=f^*(K_X+B)+E$,
    \item $B_Y\geq 0$, $E\geq 0$, and $B_Y\wedge E=0$,
    \item $(Y,B_Y)$ is lc quasi-smooth, and any lc center of $(Y,B_Y)$ on $X$ is an lc center of $(X,B)$.
    \end{enumerate}
\end{enumerate}
\end{thm}
\begin{proof} This result follows from \cite{AK00}, see also \cite[Theorem B.6]{Hu20}, \cite[Theorem 2]{Kaw15} and \cite[Step 2 of Proof of Lemma 3.2]{Has19}.
\end{proof}

\subsection{\texorpdfstring{$\bb$}{}-divisors}

\begin{defn}[$\bb$-divisors]\label{defn: b divisors} Let $X$ be a normal quasi-projective variety. We call $Y$ a \emph{birational model} over $X$ if there exists a projective birational morphism $Y\to X$. 

Let $X\dashrightarrow X'$ be a birational map. For any valuation $\nu$ over $X$, we define $\nu_{X'}$ to be the center of $\nu$ on $X'$. A \emph{$\bb$-divisor} $\Dd$ over $X$ is a formal sum $\Dd=\sum_{\nu} r_{\nu}\nu$ where $\nu$ are valuations over $X$ and $r_{\nu}\in\mathbb R$, such that $\nu_X$ is not a divisor except for finitely many $\nu$. If in addition, $r_{\nu}\in\Qq$ for every $\nu$, then $\Dd$ is called a \emph{$\Qq$-$\bb$-divisor}. The \emph{trace} of $\Dd$ on $X'$ is the $\Rr$-divisor
$$\Dd_{X'}:=\sum_{\nu_{i,X'}\text{ is a divisor}}r_i\nu_{i,X'}.$$
If $\Dd_{X'}$ is $\Rr$-Cartier and $\Dd_{Y}$ is the pullback of $\Dd_{X'}$ on $Y$ for any birational model $Y$ of $X'$, we say that $\Dd$ \emph{descends} to $X'$, and also say that $\Dd$ is the \emph{closure} of $\Dd_{X'}$, and write $\Dd=\overline{\Dd_{X'}}$. 

Let $X\rightarrow U$ be a projective morphism and assume that $\Dd$ is a $\bb$-divisor over $X$ such that $\Dd$ descends to some birational model $Y$ over $X$. If $\Dd_Y$ is nef$/U$, then we say that $\Dd$ is \emph{nef}$/U$. If $\Dd_Y$ is a Cartier divisor, then we say that $\Dd$ is \emph{$\bb$-Cartier}. If $\Dd_Y$ is a $\Qq$-Cartier $\Qq$-divisor, then we say that $\Dd$ is \emph{$\Qq$-$\bb$-Cartier}. If $\Dd$ can be written as an $\Rr_{\geq 0}$-linear combination of nef$/U$ $\bb$-Cartier $\bb$-divisors, then we say that $\Dd$ is \emph{NQC}$/U$.

We let $\bm{0}$ be the $\bb$-divisor $\bar{0}$.
\end{defn}

\begin{defn}\label{defn: b divisor restriction over an open subset}
Let $X\rightarrow U$ be a projective morphism such that $X$ is a normal quasi-projective varieties, and let $U^0$ be a non-empty open subset of $U$. Let $\Dd$ be a $\bb$-divisor over $X$. We define a $\bb$-divisor $\Dd^0:=\Dd\times_UU^0$ in the following way. 
For any birational projective morphism $Y^0\to X^0=X\times _UU^0$, we may assume that $Y^0=Y\times _UU^0$ where $Y\to X$ is a birational projective morphism. We let $\Dd^0 _{Y^0}=\Dd _Y|_{Y_0}$. It is easy to see that this definition is independent of the choice of $Y$ and defines a $\bb$-divisor.

It is easy to see that if $W\rightarrow X$ is a birational morphism such that $\Dd$ descends to $W$, then $\Dd^0$ is the closure of $\Dd_W\times_{U}U^0$. Since base change is compatible with pullbacks, $\Dd^0$ is well-defined and independent of the choice of $W$. We also note that if $\Dd$ is nef$/U$, then $\Dd^0$ is nef$/U^0$, and if $\Dd$ is NQC$/U$, then $\Dd^0$ is NQC$/U^0$.
\end{defn}

\subsection{Generalized pairs}

\begin{defn}[Generalized pairs]\label{defn: g-pairs}
A \emph{generalized sub-pair} (\emph{g-sub-pair} for short) $(X,B,\Mm)/U$ consists of a normal quasi-projective variety $X$ associated with a projective morphism $X\rightarrow U$, an $\Rr$-divisor $B$ on $X$, and a nef$/U$ $\bb$-divisor $\Mm$ over $X$, such that $K_X+B+\Mm_X$ is $\Rr$-Cartier. If $\Mm$ is NQC$/U$, then we say that $(X,B,\Mm)/U$ is an \emph{NQC g-sub-pair}. If $B$ is a $\Qq$-divisor and $\Mm$ is a $\Qq$-$\bb$-divisor, then we say that $(X,B,\Mm)/U$ is a \emph{$\Qq$-g-sub-pair}.

If $\Mm=\bm{0}$, a g-sub-pair $(X,B,\Mm)/U$ is called a \emph{sub-pair} and is denoted by $(X,B)$ or $(X,B)/U$. 

If $U=\{pt\}$, we usually drop $U$ and say that $(X,B,\Mm)$ is \emph{projective}.

A g-sub-pair (resp. NQC g-sub-pair, $\Qq$-g-sub-pair) $(X,B,\Mm)/U$ is called a \emph{g-pair} (resp. \emph{NQC g-pair}, \emph{$\Qq$-g-pair}) if $B\geq 0$. A sub-pair $(X,B)$ is called a \emph{pair} if $B\geq 0$.
\end{defn}

\begin{nota}
In the previous definition, if $U$ is not important, we may also drop $U$. This usually happens when we emphasize the structures of $(X,B,\Mm)$ that are independent of the choice of $U$, such as the singularities of $(X,B,\Mm)$. See Definition \ref{defn: sing of g-pairs} below.
\end{nota}

\begin{defn}[Singularities of generalized pairs]\label{defn: sing of g-pairs}
	Let $(X,B,\Mm)/U$ be a g-(sub-)pair. For any prime divisor $E$ and $\mathbb R$-divisor $D$ on $X$, we define $\mult_{E}D$ to be the \emph{multiplicity} of $E$ along $D$.  Let $h:W\to X$
	be any log resolution of $(X,\Supp B)$ such that $\Mm$ descends to $W$, and let
	$$K_W+B_W+\Mm_W:=h^*(K_X+B+\Mm_X).$$
	The \emph{log discrepancy} of a prime divisor $D$ on $W$ with respect to $(X,B,\Mm)$ is $1-\mult_{D}B_W$ and it is denoted by $a(D,X,B,\Mm).$
	
	We say that $(X,B,\Mm)$ is \emph{(sub-)glc} (resp. \emph{(sub-)gklt}) if $a(D,X,B,\Mm)\ge0$ (resp. $>0$) for every log resolution $h: W\to X$ as above and every prime divisor $D$ on $W$. 
	
	We say that $(X,B,\Mm)$ is \emph{gdlt} if $(X,B,\Mm)$ is glc, and there exists a closed subset $V\subset X$, such that
\begin{enumerate}
    \item $X\backslash V$ is smooth and $B_{X\backslash V}$ is simple normal crossing, and
    \item for any prime divisor $E$ over $X$ such that $a(E,X,B,\Mm)=0$, $\Center_XE\not\subset V$ and $\Center_XE\backslash V$ is an lc center of $(X\backslash V,B|_{X\backslash V})$.
\end{enumerate}
If $\Mm=\bm{0}$ and $(X,B,\Mm)$ is (sub-)glc (resp, (sub-)gklt, gdlt), we say that $(X,B)$ is (sub-)lc (resp. (sub-)klt, dlt).
	    
	 Suppose that $(X,B,\Mm)$ is sub-glc. A \emph{glc place} of $(X,B,\Mm)$ is a prime divisor $E$ over $X$ such that $a(E,X,B,\Mm)=0$. A \emph{glc center} of $(X,B,\Mm)$ is the center of a glc place of $(X,B,\Mm)$ on $X$. The \emph{non-gklt locus} $\Ngklt(X,B,\Mm)$ of $(X,B,\Mm)$ is the union of all glc centers of $(X,B,\Mm)$. If $\Mm=\bm{0}$, a glc place (resp. a glc center, the non-gklt locus) of $(X,B,\Mm)$ will be called an lc place (resp. an lc center, the non-klt locus) of $(X,B)$, and we will denote $\Ngklt(X,B,\Mm)$ by $\Nklt(X,B)$. 
	 
	 We note that the definitions above are independent of the choice of $U$.
\end{defn}

\begin{rem}\label{rem: difference in definition of g-pair}
The generalized pairs defined in \cite{BZ16} correspond to the NQC generalized pairs defined in Definition \ref{defn: g-pairs}.
\end{rem}

\begin{rem}\label{rem: difference in definition of gdlt}
The definition of gdlt in Definition \ref{defn: sing of g-pairs} is the same as the definition in \cite[Definition 2.2]{HL18}, and has slight difference with the definitions in \cite{Bir19,Fil18b,FS20b}. This definition is preserved by adjunction (cf. \cite[Proposition 2.8]{HL18}). We remark that when $X$ is $\Qq$-factorial, our definition for gdlt coincides with the definitions in \cite{Bir19,Fil18b,FS20b}. 
 
 Because of these differences in definitions, for the reader's convenience, we will usually cite \cite{HL18} for generalized pair related results although other references may have similar results.
\end{rem}

\begin{rem}\label{rem: difference in notation of g-pair}
For the notation related to generalized pairs as above, we generally adopt the same notation as in \cite{FS20b} and \cite{Has20}.  We also remark that for log discrepancies of generalized pairs, we use the notation $a(D,X,B,\Mm)$ instead of $a(D,X,B+\Mm_X)$. This is because $(X,B+\Mm_X)$ is a sub-pair, and the log discrepancy of the sub-pair $(X,B+\Mm_X)$ may not be equal to the log discrepancy of the generalized pair $(X,B,\Mm)$. Our notation is also similar to the notation as in \cite{HL19,HL20d} where they use $(X/U,B+\Mm)$ for generalized pairs and $a_D(X/U,B+\Mm)$ for log discrepancies. We do not use their notation as well because it is important to notice that log discrepancies of a generalized pair are independent of the base $U$.
\end{rem}

\subsection{Some results on MMPs for generalized pairs}

\subsubsection{Set-up} 

\begin{defn}
Let $(X,B,\Mm)/U$ be a glc g-pair. We say that a $(K_X+B+\Mm_X)$-MMP$/U$: $X\dashrightarrow X'$ \emph{ends with a minimal model} if $K_{X'}+B'+\Mm_{X'}$ is nef$/U$, where $B'$ is the strict transform of $B$ on $X'$. We say that  a $(K_X+B+\Mm_X)$-MMP$/U$: $X\dashrightarrow X'$ \emph{ends with a Mori fiber space} if there exists a $(K_{X'}+B'+\Mm_{X'})$-Mori fiber space structure $X'\rightarrow Z$ over $U$. 
\end{defn}

The following result tells us that we can always run the MMP for generalized lc pairs $(X,B,\Mm)/U$ when $X$ is $\Qq$-factorial klt.
\begin{thm}[{cf. \cite[Lemma 3.5]{HL18}}]\label{thm: can run mmp for gklt pair}
Let $(X,B,\Mm)/U$ be a $\Qq$-factorial glc g-pair such that $X$ is klt. Then we can always run a $(K_X+B+\Mm_X)$-MMP$/U$. More precisely, there exists a sequence of flips and divisorial contractions over $U$, which ends either with a Mori fiber space, or a minimal model, or an infinite sequence of flips over $U$.
\end{thm}

We also make the following remark on MMPs with scaling of relatively ample $\Rr$-divisors:

\begin{rem}
Let $(X,B,\Mm)/U$ be a $\Qq$-factorial glc g-pair such that $X$ is klt. When we say ``we run a $(K_X+B+\Mm_X)$-MMP$/U$ with scaling of an ample$/U$ $\Rr$-divisor", we always assume that the choice of the ample$/U$ $\Rr$-divisor $A$ on $X$ satisfies that $A\geq 0$, $(X,B+A,\Mm)$ is glc, and $K_X+B+A+\Mm_X$ is nef$/U$. We remark that for any ample$/U$ $\Rr$-divisor $A$ on $X$, one can always choose $A'\sim_{\Rr,U}A$ such that $(X,B+A',\Mm)$ is glc.
\end{rem}

\subsubsection{MMP for very exceptional divisors}

\begin{lem}[MMP for very exceptional divisors, {cf. \cite[Proposition 3.8]{HL18}}]\label{lem: rlinear version of hl18 3.8}
Let $(X,B,\Mm)/U$ be a $\Qq$-factorial glc g-pair such that $X$ is klt and $K_X+B+\Mm_X\equiv_{U}D_1-D_2$ (resp.  $\sim_{\mathbb R,U}D_1-D_2$) where $D_1\geq 0$, $D_2\geq 0$ have no common components. Suppose that $D_1$ is very exceptional over $U$. Then any $(K_X+B+\Mm_X)$-MMP$/U$ with scaling of an ample$/U$ $\Rr$-divisor either terminates with a Mori fiber space or contracts $D_1$ after finitely many steps. Moreover, if $D_2=0$, then this MMP terminates with a model $Y$ such that  $K_Y+B_Y+\Mm_Y\equiv_{U}0$ (resp. $\sim_{\mathbb R,U}0$), where $B_Y$ is the strict transform of $B$ on $Y$. 
\end{lem}
\begin{proof}
The numerical equivalence part of the lemma is exactly \cite[Proposition 3.8]{HL18}. Thus we we assume that $K_X+B+\Mm_X\sim_{\mathbb R,U}D_1-D_2$, $D_2=0$, and we only need to show that the MMP terminates with a model $Y$ such that $K_Y+B_Y+\Mm_Y\sim_{\mathbb R,U}0$, where $B_Y$ is the strict transform of $B$ on $Y$.

By the numerical equivalence part of the lemma, the MMP contracts $D_1$ after finitely many steps. We may let $\phi: X\dashrightarrow Y$ be the birational map corresponding to this partial MMP. Since $K_X+B+\Mm_X\sim_{\mathbb R,U}D_1$ and $\phi$ contracts $D_1$, we have $K_Y+B_Y+\Mm_Y\sim_{\mathbb R,U}0$, where $B_Y$ is the strict transform of $B$ on $Y$, and the lemma is proved.
\end{proof}

\subsubsection{MMP with scaling and log minimal models}

We refer the reader to \cite[Lemma 3.19, Definition 3.20]{HL18} for the definition  MMP with scaling for generalized pairs.

We need the following definition for log minimal models. A detailed discussion of log minimal models and other models for g-pairs will be given in Section 3.

\begin{defn}[Log minimal models, cf. Definition {\ref{defn: models}(3)} and {\cite[Definition 2.9]{HL18}}]
Let $(X,B,\Mm)/U$ be a glc g-pair, $\phi: X\dashrightarrow X'$ a birational map over $U$, and $E:=\Exc(\phi^{-1})$ the reduced $\phi^{-1}$-exceptional divisor. A g-pair $(X',B',\Mm)/U$ is called a \emph{log minimal model} of $(X,B,\Mm)/U$ if
\begin{enumerate}
    \item $B'=\phi_*B+E$,
    \item $K_{X'}+B'+\Mm_{X'}$ is nef$/U$,
    \item $(X',B',\Mm)$ is $\Qq$-factorial gdlt, and
    \item for any prime divisor $D$ on $X$ which is exceptional over $X'$, $a(D,X,B,\Mm)<a(D,X',B',\Mm)$.
\end{enumerate}
\end{defn}

We need the following theorem from \cite{BZ16}:

\begin{thm}[{cf. \cite[Theorem 4.4(2)]{BZ16}}]\label{thm: bz16 4.4(2)}
Let $(X,B,\Mm)/U$ be a $\Qq$-factorial NQC gklt g-pair and $A\geq 0$ an ample$/U$ $\Rr$-divisor on $X$, such that 
\begin{enumerate}
    \item $K_X+B+\Mm_X$ is pseudo-effective$/U$,
    \item $K_X+B+\Mm_X+(1+\alpha)B+(1+\beta)\Mm_X$ is big$/U$ for some $\alpha,\beta\geq 0$,
    \item $(X,B+A,\Mm)$ is glc and $K_X+B+A+\Mm_X$ is nef$/U$.
\end{enumerate}
Then we can run a $(K_X+B+\Mm_X)$-MMP$/U$ with scaling of $A$, which terminates with a log minimal model $(X',B',\Mm)/U$ such that $K_{X'}+B'+\Mm_{X'}$ is semi-ample$/U$.
\end{thm}
\begin{proof}
We may pick a real number $0<\epsilon\ll 1$ such that $(1+\epsilon)(K_X+B+\Mm_X)\sim_{\Rr,U}K_X+\Delta$ for some klt pair $(X,\Delta)$ such that $\Delta$ is big$/U$ (see the proof of \cite[Theorem 4.4(2)]{BZ16}). Then any $(K_X+B+\Mm_X)$-MMP$/U$ with scaling of $A$ is also a $(K_X+\Delta)$-MMP$/U$ with scaling of $A'\sim_{\Rr}(1+\epsilon)A$ for some ample$/U$ $\Rr$-divisor $A'\geq 0$ such that $(X,\Delta+A')$ is klt. By \cite[Corollary 1.4,2]{BCHM10}, this MMP terminates with a log minimal model $X'$ such that $K_{X'}+\Delta'$ is semi-ample$/U$, where $\Delta'$ is the strict transform of $\Delta$ on $\Delta'$. Thus $(X',B',\Mm)/U$ is a log minimal model of $(X,B,\Mm)/U$ such that $K_{X'}+B'+\Mm_{X'}$ is semi-ample$/U$, where $B'$ is the strict transform of $B$ on $X'$.
\end{proof}

\begin{thm}[{\cite[Remark 3.21, Theorem 4.1]{HL18}}]\label{thm: hl18 4.1}
Let $(X,B,\Mm)/U$ be a $\Qq$-factorial NQC glc g-pair such that $X$ is klt, $D\geq 0$ an $\Rr$-divisor on $X$, and $\NN$ an NQC$/U$ $\bb$-divisor over $X$, such that $(X,B+D,\Mm+\NN)$ is glc and $K_X+B+D+\Mm_X+\NN_X$ is nef$/U$. Assume that there exists a $(K_X+B+\Mm_X)$-MMP$/U$ with scaling of $D+\NN_X$:
$$(X,B,\Mm):=(X_1,B_1,\Mm)\dashrightarrow (X_2,B_2,\Mm)\dashrightarrow\dots\dashrightarrow (X_i,B_i,\Mm)\dashrightarrow\dots,$$
and let $\lambda_i$ be the $i$-th scaling number of this MMP for each $i$, i.e.
$$\lambda_i:=\inf\{t\mid t\geq 0, K_{X_i}+B_i+tD_i+\Mm_{X_i}+t\NN_{X_i}\text{ is nef/}U\},$$
where $D_i$ is the strict transform of $D$ on $X_i$. Then $\lambda_i\geq\lambda_{i+1}$ for each $i$, and one of the following holds:
\begin{enumerate}
    \item This MMP terminates after finitely many steps.
    \item This MMP does not terminate and $\lambda_i=\lambda_{i+1}$ for any $i\gg 0$.
    \item This MMP does not terminate,  $\lambda:=\lim_{i\rightarrow+\infty}\lambda_i\not=\lambda_j$ for any $j$, and $(X,B+\lambda D,\Mm+\lambda\NN)/U$ does not have a log minimal model.
\end{enumerate}
\end{thm}

\begin{thm}\label{thm: mmp with scaling gpair terminates assuming gmm}
Let $(X,B,\Mm)/U$ be a $\Qq$-factorial NQC glc g-pair such that $X$ is klt, and $A\geq 0$ an ample$/U$ $\Rr$-divisor on $X$ such that $(X,B+A,\Mm)$ is glc and $K_X+B+A+\Mm_X$ is nef$/U$. Let $$(X,B,\Mm):=(X_1,B_1,\Mm)\dashrightarrow (X_2,B_2,\Mm)\dashrightarrow\dots\dashrightarrow (X_i,B_i,\Mm)\dashrightarrow\dots$$
be a $(K_X+B+\Mm_X)$-MMP$/U$ with scaling of $A$, and let $\lambda_i$ be the $i$-th scaling number of this MMP for each $i$, i.e.
$$\lambda_i:=\inf\{t\mid t\geq 0, K_{X_i}+B_i+tA_i+\Mm_{X_i}\text{ is nef/}U\},$$
where $A_i$ is the strict transform of $A$ on $X_i$ for each $i$. Then $\lambda_i\geq\lambda_{i+1}$ for each $i$, and one of the following holds:
\begin{enumerate}
    \item This MMP terminates after finitely many steps.
    \item $\lim_{i\rightarrow +\infty}\lambda_i=0$, and $(X,B,\Mm)$ does not have a log minimal model.
\end{enumerate}
In particular, if $(X,B,\Mm)/U$ is gdlt and has a log minimal model, then this MMP terminates with log minimal model of $(X,B,\Mm)/U$.
\end{thm}
\begin{proof}
By Theorem \ref{thm: hl18 4.1},  $\lambda_i\geq\lambda_{i+1}$ for each $i$, and we may assume that this MMP does not terminate and $\lambda_i=\lambda_{i+1}>0$ for any $i\gg 0$. Let $\lambda:=\lim_{i\rightarrow+\infty}\lambda_i$, then $\lambda>0$. Since $X$ is $\Qq$-factorial klt, by \cite[Lemma 3.5]{HL18}, we may pick 
$$0\leq \Delta\sim_{\mathbb R,U}B+\Mm_X+\frac{\lambda}{2}A$$ 
such that $(X,\Delta)$ is klt and $\Delta$ is big$/U$. Now this MMP is also a $(K_X+\Delta)$-MMP with scaling of $0\leq A'\sim_{\mathbb R,U}(1-\frac{\lambda}{2})A$ for some $A'$ such that $(X,\Delta+A')$ is klt. This MMP terminates by \cite[Corollary 1.4.2]{BCHM10}, a contradiction.

The in particular part follows from the fact that $(X_i,B_i,\Mm)$ is $\Qq$-factorial gdlt for each $i$ if $(X,B,\Mm)$ is gdlt, and $a(D,X,B,\Mm)<a(D,X_i,B_i,\Mm)$ for any $i$ and any prime divisor $D$ on $X$ that is exceptional over $X_i$.
\end{proof}

\subsubsection{Perturbation of MMPs}

\begin{lem}\label{lem: still an mmp under perturbation}
Let $X\rightarrow U$ be a projective morphism such that $X$ is normal quasi-projective. Let $D,A$ be two $\Rr$-Cartier $\Rr$-divisors on $X$ and let $\phi: X\dashrightarrow X'$ be a partial $D$-MMP$/U$. Then there exists a positive real number $t_0$, such that for any $t\in (0,t_0]$, $\phi$ is also a partial $(D+tA)$-MMP$/U$. Note that $A$ is not necessarily effective.
\end{lem}
\begin{proof}
We let
$$X:=X_1\dashrightarrow X_2\dashrightarrow\dots\dashrightarrow X_n=X'$$
be this partial MMP, and $D_i,A_i$ the strict transforms of $D$ and $A$ on $X_i$ respectively. Let $X_i\rightarrow Z_i$ be the $D_i$-negative extremal contraction of a $D_i$-negative extremal ray $R_i$ in this MMP for each $i$, then $D_i\cdot R_i<0$ for each $i$. Thus there exists a positive real number $t_0$, such that $(D_i+t_0A_i)\cdot R_i<0$ for each $i$. In particular, $(D_i+tA_i)\cdot R_i<0$ for any $i$ and any $t\in (0,t_0]$. Thus $\phi$ is a partial $(D+tA)$-MMP$/U$ for any $t\in (0,t_0]$.
\end{proof}

\begin{lem}[{\cite[Lemma 3.17]{HL18}}]\label{lem: trivial mmp under perturbation}
Let $(X,B+A,\Mm)/U$ be a $\Qq$-factorial NQC glc g-pair such that $X$ is klt, $(X,B,\Mm)$ is glc, and $K_X+B+\Mm_X$ is nef$/U$. Then there exists a positive real number $t_0$, such that for any $t\in (0,t_0]$, any partial $(K_X+B+tA+\Mm_X)$-MMP$/U$ is $(K_X+B+\Mm_X)$-trivial. Note that $A$ is not necessarily effective.
\end{lem}

\subsection{Rational polytopes for generalized pairs}

In some situations, we need to perturb the coefficients of NQC g-pairs in order to use the results for $\Qq$-g-pairs.  The key ideas are simple: First, we have a rational polytope (Shokurov-type polytope) for NQC glc pairs with nef generalized log canonical divisors (\cite[3.3]{HL18}). Second, for usual pairs, Han-Liu-Shokurov establishes a complete state-of-the-art theory for rational polytopes \cite[Section 5]{HLS19} with many important applications in birational geometry. Therefore, we will adopt the ideas in \cite{HLS19} to prove Theorem \ref{thm: shokurov polytope gpair} (cf. \cite[3.3]{HL18}) which also addresses some additional properties of generalized pairs. We remark that although there are many works and results in this direction (\cite{HL18,HL19,HL20d,Che20}), directly applying these results is not sufficient for our purposes.

First we prove an easy lemma.

\begin{lem}\label{lem: from polytope to perturbation}
Let $n,c$ be two non-negative integers, and let $\bm{v}_1,\dots,\bm{v}_{c+1}\in\mathbb Q^n$ be $c+1$ rational points. Let $\bm{v}\in\mathbb R^n$ is a point which is contained in the interior of the convex hull of $\bm{v}_1,\dots,\bm{v}_{c+1}$. Then there exist real numbers $a_1,\dots,a_{c+1}\in (0,1]$, such that $\sum_{i=1}^{c+1}a_i=1$ and $\sum_{i=1}^{c+1}a_i\bm{v}_i=\bm{v}$.
\end{lem}
\begin{proof}
Let $\bm{v}_i':=\bm{v}_i-\bm{v}$ for each $i$. Then $\bm{0}$ is contained in the interior of the convex hull spanned by $\bm{v}_1',\dots,\bm{v}_{c+1}'$. Thus there exist positive real numbers $a_1',\dots,a_{c+1}'$ such that $\sum_{i=1}^{c+1}a_i'\bm{v}_i'=\bm{0}$. We may let $a_i:=\frac{a_i'}{\sum_{j=1}^{c+1}a_j'}$ for each $i$. 
\end{proof}

Some notation in the next theorem is taken from \cite{HLS19} and \cite{Nak16}.
\begin{thm}\label{thm: shokurov polytope gpair}
Let $c,m,n,l$ be four non-negative integers, $r_1,\dots,r_c$ real numbers such that $1,r_1,\dots,r_c$ are linearly independent over $\Qq$, $\bm{r}:=(r_1,\dots,r_c)$, and $s_1,\dots,s_{m+n}:\mathbb R^{c+1}\rightarrow\mathbb R$ are $\Qq$-linear functions, such that $s_j(1,\bm{r})\geq 0$ for any $j$.

Let $(X,B,\Mm)/U$ be an NQC glc g-pair, $B=\sum_{j=1}^ms_j(1,\bm{r})B_j$, and $\Mm=\sum_{j=1}^ns_{j+m}(1,\bm{r})\Mm_j$, where each $B_j\geq 0$ is a $\Qq$-divisor and each $\Mm_j$ is a nef$/U$ $\Qq$-$\bb$-divisor. Let $B(\bm{x}):=\sum_{j=1}^ms_j(1,\bm{x})B_j$ and $\Mm(\bm{x}):=\sum_{j=1}^ns_{j+m}(1,\bm{x})\Mm_j$ for any $\bm{x}\in\mathbb R^c$. Let $U^0\subset U$ be a non-empty open subset, $X^0:=X\times_UU^0$, and let $S_1,\dots,S_l$ be the normalization of the irreducible components of $\lfloor B\rfloor$.

Then there exists an open set $V\ni\bm{r}$ of $\mathbb R^c$ (which may depend on $(X,B,\Mm)/U$, etc.) satisfying the following. For any $\bm{v}\in V$,
\begin{enumerate}
\item $s_j(1,\bm{v})\geq 0$ for each $j$,
\item $(X,B(\bm{v}),\Mm(\bm{v}))/U$ is an NQC glc g-pair,
\item $\Ngklt(X,B(\bm{v}),\Mm(\bm{v}))=\Ngklt(X,B,\Mm)$,
\item if $X\rightarrow Z$ is a projective surjective morphism over $U$ such that $K_X+B+\Mm_X\sim_{\Rr,Z}0$, then $K_X+B(\bm{x})+\Mm(\bm{x})_X\sim_{\Rr,Z}0$ for any $\bm{x}\in\mathbb R^c$,
\item if $\Mm_X|_{X^0}\sim_{\Rr,U^0}0$, then $\Mm(\bm{x})_X|_{X^0}\sim_{\Rr,U^0}0$ for any $\bm{x}\in\mathbb R^c$,
\item if $\Mm\times_UU^0$ descends to $X^0$, then $\Mm(\bm{x})\times_UU^0$ descends to $X^0$ for any $\bm{x}\in\mathbb R^c$, 
\item if $K_X+B+\Mm_X$ is nef$/U$, then $K_X+B(\bm{v})+\Mm(\bm{v})_X$ is nef$/U$,
\item if $(K_X+B+\Mm_X)|_{X^0}$ is semi-ample$/U^0$, then $(K_X+B(\bm{v})+\Mm(\bm{v})_X)|_{X^0}$ is semi-ample$/U^0$,
\item for any $k$, if $(K_X+B+\Mm_X)|_{S_k}$ is semi-ample$/U$, then $(K_X+B(\bm{v})+\Mm(\bm{v})_X)|_{S_k}$ is semi-ample$/U$.
\end{enumerate}
In particular, there exist positive real numbers $a_1,\dots,a_{c+1}\in (0,1]$ and $\Qq$-g-pairs $(X,B^i,\Mm^i)/U$, such that
\begin{itemize}
    \item $\sum_{i=1}^{c+1}a_i=1$, $B=\sum_{i=1}^{c+1}a_iB^i$, and $\Mm=\sum_{i=1}^{c+1}a_i\Mm^i$,
    \item $(X,B^i,\Mm^i)/U$ is glc for each $i$,
    \item $\Ngklt(X,B^i,\Mm^i)=\Ngklt(X,B,\Mm)$ for each $i$,
    \item if $X\rightarrow Z$ is a projective surjective morphism over $U$ such that $K_X+B+\Mm_X\sim_{\Rr,Z}0$, then $K_X+B^i+\Mm^i_X\sim_{\Qq,Z}0$ for each $i$,
    \item if $\Mm_X|_{X^0}\sim_{\Rr,U^0}0$, then $\Mm^i_X|_{X^0}\sim_{\Qq,U^0}0$ for each $i$,
\item if $\Mm\times_UU^0$ descends to $X^0$, then $\Mm^i\times_UU^0$ descends to $X^0$ for each $i$,
\item if $K_X+B+\Mm_X$ is nef$/U$, then $K_X+B^i+\Mm^i_X$ is nef$/U$ for each $i$,
\item if $(K_X+B+\Mm_X)|_{X^0}$ is semi-ample$/U^0$, then $(K_X+B^i+\Mm^i_X)|_{X^0}$ is semi-ample$/U^0$ for each $i$,
\item for any $k$, if $(K_X+B+\Mm_X)|_{S_k}$ is semi-ample$/U$, then $(K_X+B^i+\Mm^i_X)|_{S_k}$ is semi-ample$/U$ for each $i$.
\end{itemize}
\end{thm}
\begin{proof}
We only need to find open subsets $V$ satisfying each individual condition and then take their common intersections. (1) is obvious. (2) and (3) follow from \cite[Theorem 1.4]{Che20} and the linearity of log discrepancies. (4) and (5) follow from the proof of \cite[Lemma 5.3]{HLS19}. 

To prove (6), let $\Mm^0:=\Mm\times_UU^0$ and $\Mm(\bm{x})^0:=\Mm(\bm{x})\times_UU^0$ for any $\bm{x}\in\mathbb R^c$. Let $g: \tilde X\rightarrow X$ be a resolution such that $\Mm$ and any $\Mm_j$ descend to $\tilde X$, and let $\tilde X^0:=\tilde X\times_UU^0$. Since $\Mm^0$ descends to $X^0$, $\Mm^0_{X^0}$ is $\Rr$-Cartier. By \cite[Lemma 5.3]{HLS19}, $\Mm(\bm{x})^0_{X^0}$ is $\Rr$-Cartier for any $\bm{x}\in\mathbb R^c$. By the $\Qq$-linearity of log discrepancies, there exists a $\Qq$-affine function $F: \mathbb R^c\rightarrow\Div_{\mathbb R}(\tilde X^0)$, such that for any $\bm{x}\in\mathbb R^c$,
$$\Mm(\bm{x})^0_{\tilde X^0}=(g|_{\tilde X^0})^*\Mm(\bm{x})^0_{X^0}+F(\bm{x}),$$
where $F(\bm{x})$ is exceptional over $X^0$. Since $\Mm^0$ descends to $X^0$, $F(\bm{r})=0$. Thus $F(\bm{x})=0$ for any $\bm{x}\in\mathbb R^c$, hence $\Mm(\bm{x})^0$ descends to $X^0$ for any $\bm{x}\in\mathbb R^c$.

We prove (7). Let $f: W\rightarrow X$ be a gdlt modification of $(X,B,\Mm)$ (see Definition-Lemma \ref{deflem: gdlt modification} below) such that
$$K_W+f^{-1}_*B+E+\Mm_W=f^*(K_X+B+\Mm_X)$$
where $E:=\Exc(f)$. Let $B_{j,W}$ be the strict transform of $B_j$ on $W$ for each $j$, and $B_W(\bm{x}):=\sum_{j=1}^ms_j(1,\bm{x})B_{j,W}+E$ for any $\bm{x}\in\mathbb R^c$. By $\Qq$-linearity of log discrepancies, we have
$$K_W+B_W(\bm{x})+\Mm(\bm{x})_W=f^*(K_X+B(\bm{x})+\Mm(\bm{x})_X)$$
for any $\bm{x}\in\mathbb R^c$. By \cite[Proposition 3.16]{HL18}, there exists an open subset $V\ni\bm{r}$ such that $K_W+B_W(\bm{v})+\Mm(\bm{v})_W$ is nef$/U$ for any $\bm{v}\in V$, and this $V$ satisfies (7).

We prove (8). If $(K_X+B+\Mm_X)|_{X^0}$ is semi-ample$/U^0$, then we let $\phi: X^0\rightarrow Y^0$ be the contraction defined by $(K_X+B+\Mm_X)|_{X^0}$ over $U^0$. We have
$$(K_X+B+\Mm_X)|_{X^0}\sim_{\mathbb R,U^0}\phi^*A$$
for some ample$/U^0$ $\Rr$-divisor $A$ on $Y^0$. By (5), for any $\bm{x}\in\mathbb R^c$, we have $(K_X+B(\bm{x})+\Mm(\bm{x})_X)|_{X^0}\sim_{\mathbb R,Y^0}0$, so $(K_X+B(\bm{x})+\Mm(\bm{x})_X)|_{X^0}\sim_{\mathbb R,U^0}\phi^*A(\bm{x})$ for some $\Rr$-divisor $A(\bm{x})$ on $Y^0$. Possibly replacing $A(\bm{x})$, we may assume that $\mathbb R^c\rightarrow\Div_{\mathbb R}(Y^0)$ given by $\bm{x}\rightarrow A(\bm{x})$ is an affine function and $A(\bm{r})=A$. Since ampleness is an open condition, there exists an open set $V\ni\bm{r}$ such that $A(\bm{v})$ is ample$/U^0$ for any $\bm{v}\in V$. 

We prove (9), and finish the proof of the main part of the theorem. For any $k$, let $(S_k,B_{S_k},\Mm^{S_k})$ be the NQC glc g-pair given by the adjunction
$$K_{S_k}+B_{S_k}+\Mm^{S_k}_{S_k}:=(K_X+B+\Mm_X)|_{S_k},$$
then we may write $B_{S_k}=\sum_{j=1}^{m_k}s^k_j(1,\bm{r})B_{j,S_k}$ and $\Mm^{S_k}=\sum_{j=1}^{n_k}s_{j+m_k}(1,\bm{r})\Mm^{S_k}_j$, where $B_{j,S_k}\geq 0$ are $\Qq$-divisors, $\Mm_j^{S_k}$ are nef$/U$ $\Qq$-$\bb$-divisors, and $s_j^k(1,\bm{r})\geq 0$ for any $j,k$. Let $B_{S_k}(\bm{x}):=\sum_{j=1}^{m_k}s^k_j(1,\bm{x})B_{j,S_k}$ and $\Mm^{S_k}(\bm{x}):=\sum_{j=1}^{n_k}s_{j+m_k}(1,\bm{x})\Mm^{S_k}_j$ for any $\bm{x}\in\mathbb R^c$, then
$$K_{S_k}+B_{S_k}(\bm{x})+\Mm^{S_k}_{S_k}(\bm{x}):=(K_X+B(\bm{x})+\Mm_X(\bm{x}))|_{S_k}$$
for any $\bm{x}\in\mathbb R^c$. Thus (9) follows from (8).

To prove the in particular part of the theorem, we let $\bm{v}_1,\dots,\bm{v}_{c+1}\in V\cap\mathbb Q^c$ be $c+1$ points such that $\bm{r}$ is contained in the interior of the convex hull of $\bm{v}_1,\dots,\bm{v}_{c+1}$. By Lemma \ref{lem: from polytope to perturbation}, we may let $B^i:=B(\bm{v}_i)$ and $\Mm^i:=\Mm(\bm{v}_i)$ for each $i$, and let $a_1,\dots,a_{c+1}\in (0,1]$ be any real numbers such that $\sum_{i=1}^{c+1}a_i=1$ and $\sum_{i=1}^{c+1}a_i\bm{v}_i=\bm{r}$.
\end{proof}

\section{Models}

In this sections, we will study  different types of models of generalized pairs. For the case of models of usual pairs, we refer the reader to \cite[Section 2]{Bir12a}, \cite[Section 2]{Has19}.

\subsection{Definitions}

\begin{defn}[Log smooth model]\label{defn: log smooth models}
Let $(X,B,\Mm)/U$ be a glc g-pair and $h: W\rightarrow X$ a log resolution of $(X,\Supp B)$ such that $\Mm$ descends to $W$. Let $B_W\geq 0$ and $E\geq 0$ be two $\Rr$-divisors on $W$ such that
\begin{enumerate}
    \item $K_W+B_W+\Mm_W=h^*(K_X+B+\Mm_X)+E$,
    \item $(W,B_W)$ is log smooth dlt,
    \item $E$ is $h$-exceptional, and
    \item for any $h$-exceptional prime divisor $D$ such that $a(D,X,B,\Mm)>0$, $D$ is a component of $E$.
\end{enumerate}
Then $(W,B_W,\Mm)$ is called a \emph{log smooth model} of $(X,B,\Mm)$. If we additionally assume that
\begin{itemize}
    \item[(5)] for any $h$-exceptional prime divisor $D$ such that $a(D,X,B,\Mm)>0$, $D$ is a component of $\{B_W\}$,
\end{itemize}
then $(W,B_W,\Mm)$ is called a \emph{proper log smooth model} of $(X,B,\Mm)$.
\end{defn}

\begin{defn}[Models]\label{defn: models}
Let $(X,B,\Mm)/U$ be a glc g-pair, $\phi: X\dashrightarrow X'$ a proper birational map over $U$, and $E:=\Exc(\phi^{-1})$ the reduced $\phi^{-1}$-exceptional divisor. Let $B':=\phi_*B+E$.

\begin{enumerate}
    \item $(X',B',\Mm)/U$ is called a \emph{log birational model} of $(X,B,\Mm)/U$. 
    \item $(X',B',\Mm)/U$ is called a \emph{weak glc model} of $(X,B,\Mm)/U$ if
\begin{enumerate}
\item $(X',B',\Mm)/U$ is a log birational model of $(X,B,\Mm)/U$, 
    \item $K_{X'}+B'+\Mm_{X'}$ is nef$/U$, and
    \item for any prime divisor $D$ on $X$ which is exceptional over $X'$, $a(D,X,B,\Mm)\leq a(D,X',B',\Mm)$.
\end{enumerate}
\item $(X',B',\Mm)/U$ is called a \emph{log minimal model} of $(X,B,\Mm)/U$ if
\begin{enumerate}
    \item $(X',B',\Mm)/U$ is a weak glc model of $(X,B,\Mm)/U$,
    \item $(X',B',\Mm)$ is $\Qq$-factorial gdlt, and
    \item for any prime divisor $D$ on $X$ which is exceptional over $X'$, $a(D,X,B,\Mm)<a(D,X',B',\Mm)$.
\end{enumerate}
\item 
$(X',B',\Mm)/U$ is called a \emph{good minimal model} of $(X,B,\Mm)/U$ if
\begin{enumerate}
        \item $(X',B',\Mm)/U$ is a log minimal model of $(X,B,\Mm)/U$, and
        \item $K_{X'}+B'+\Mm_{X'}$ is semi-ample$/U$.
\end{enumerate}
\end{enumerate}
\end{defn}

\begin{deflem}[Gdlt modification, {\cite[Proposition 3.9]{HL18}}]\label{deflem: gdlt modification}
Let $(X,B,\Mm)/U$ be a glc g-pair. Then there exists a birational morphism $f: Y\rightarrow X$ and a glc g-pair $(Y,B_Y,\Mm)/U$, such that
\begin{enumerate}
\item $(Y,B_Y,\Mm)$ is $\Qq$-factorial gdlt,
    \item $K_Y+B_Y+\Mm_Y=f^*(K_X+B+\Mm_X)$, and
    \item any $f$-exceptional divisor is a component of $\lfloor B_Y\rfloor$.
\end{enumerate}
For any birational morphism $f$ and $(Y,B_Y,\Mm)$ which satisfies (1-3), $f$ will be called a \emph{gdlt modification} of $(X,B,\Mm)$, and $(Y,B_Y,\Mm)$ will be called a  \emph{gdlt model} of $(X,B,\Mm)$.
\end{deflem}

\begin{rem}
As the definition of gdlt follows from \cite{HL18} and is different from \cite{Bir20a} and \cite{FS20b}, we do not define the gdlt modifications for g-pairs that are not glc here.
\end{rem}

\begin{rem}
Log birational models, weak glc models, log minimal models, and good minimal models depend on the base $U$, so in the definitions, we have the notation ``$/U$". On the other hand, log smooth models and gdlt models do not depend on the base $U$, so in the definitions, we do not have the notation ``$/U$". 
\end{rem}

\subsection{Proper log smooth models}

\begin{lem}\label{lem: existence of proper log smooth model}
Let $(X,B,\Mm)/U$ be a glc g-pair and $h: W\rightarrow X$ a log resolution of $(X,\Supp B)$ such that $\Mm$ descends to $W$. Then $(X,B,\Mm)$ has a proper log smooth model $(W,B_W,\Mm)$ for some $\Rr$-divisor $B_W$ on $W$.
\end{lem}
\begin{proof}
Assume that
$$K_W+h^{-1}_*B+\Gamma+\Mm_W=h^*(K_X+B+\Mm_X),$$
then $\Gamma$ is $h$-exceptional. Let $E=\Exc(h)$ be the reduced $h$-exceptional divisor. Then there exists a real number $\epsilon\in (0,1)$, such that for any component $D$ of $E$, if $\mult_D\Gamma<1$, then $\mult_D\Gamma<1-\epsilon$. We let $$B_W:=h^{-1}_*B+\epsilon\Gamma^{=1}+(1-\epsilon)E,$$
then $(W,B_W,\Mm)$ is a proper log smooth model of $(X,B,\Mm)$.
\end{proof}

\begin{lem}\label{lem: proper log smooth model keep lc center}
Let $(X,B,\Mm)/U$ be a glc g-pair and $(W,B_W,\Mm)$ a proper log smooth model of $(X,B,\Mm)$ with induced morphism $h: W\rightarrow X$. Assume that
$$K_W+B_W+\Mm_W=h^*(K_X+B+\Mm_X)+E,$$
then:
\begin{enumerate}
    \item $\Supp B_W=\Supp h^{-1}_*B\cup\Exc(h)$.
    \item For any prime divisor $D$ on $W$ that is exceptional over $X$, $D$ is a component $E$ if and only if $a(D,X,B,\Mm)>0$.
    \item Any glc place of $(W,B_W,\Mm)$ is a glc place of $(X,B,\Mm)$. In particular, the image of any glc center of $(W,B_W,\Mm)$ on $X$ is a glc center of $(X,B,\Mm)$.
\end{enumerate}
\end{lem}
\begin{proof}
First we prove (1). By construction, $\Supp B_W\subset\Supp h^{-1}_*B\cup\Exc(h)$ and $\Supp h^{-1}_*B\subset\Supp B_W$. Let $D$ be a component of $\Exc(h)$. If $a(D,X,B,\Mm)=0$, then since $E\geq 0$, $D$ is a component of $B_W$. If $a(D,X,B,\Mm)>0$, by Definition \ref{defn: log smooth models}(5), $E$ is a component of $\{B_W\}$, hence a component of $B_W$. Thus $\Exc(h)\subset\Supp B_W$, and we have (1).

We prove (2). Let $D$ be a prime divisor on $W$. If $a(D,X,B,\Mm)>0$, then $D$ is a component of $E$ by Definition \ref{defn: log smooth models}(4). If $a(D,X,B,\Mm)=0$, then
$$0=a(D,X,B,\Mm)=a(D,W,B_W-E,\Mm)\geq a(D,W,B_W,\Mm)\geq 0,$$
which implies that $a(D,W,B_W-E,\Mm)=a(D,W,B_W,\Mm)$, hence $\mult_DE=0$. Thus we have (2).

We prove (3). Let $D$ be a glc place of $(W,B_W,\Mm)$. Then the center of $D$ on $W$ is a stratum of $\lfloor B_W\rfloor$. If $\Center_WD\subset\Supp E$, then since $B_W+E$ is simple normal crossing, there exists a prime divisor $F$ that is a component of $\lfloor B_W\rfloor$ such that $\Center_WD\subset F$ and $F$ is a component of $E$. By (2), $a(F,X,B,\Mm)>0$. By Definition \ref{defn: log smooth models}(5), $F$ is a component of $\{B_W\}$, so $F$ cannot be a component of $\lfloor B_W\rfloor$, a contradiction. Thus $\Center_WD\not\subset\Supp E$. Therefore, any glc place of $(W,B_W,\Mm)$ is a glc place of  of $(W,B_W-E,\Mm)$, hence a glc place of $(X,B,\Mm)$, and we have (3).
\end{proof}

\subsection{Models under some birational maps}

\subsubsection{Models under resolutions}

\begin{lem}\label{lem: g-pair version bir12 2.6}
Let $(X,B,\Mm)/U$ be a glc g-pair, $(X',B',\Mm)/U$ a weak glc model of $(X,B,\Mm)/U$ with birational map $\phi: X\dashrightarrow X'$, and $p: W\rightarrow X$ and $q: W\rightarrow X'$ a common resolution of $(X,B,\Mm)$ and $(X',B',\Mm)$ such that $q=\phi\circ p$. Assume that
$$p^*(K_X+B+\Mm_X)=q^*(K_{X'}+B'+\Mm_{X'})+E,$$
then
\begin{enumerate}
    \item $E\geq 0$, and
    \item $E$ is exceptional over $X'$.
\end{enumerate}
\end{lem}
\begin{proof}
For any prime divisor $D$ that is an irreducible component of $E$, $$\mult_DE=a(D,X',B',\Mm)-a(D,X,B,\Mm).$$ Thus if $D$ is not exceptional over $X$, then
\begin{itemize}
    \item if $D$ is not exceptional over $X'$, then $\mult_DE=0$, and
    \item if $D$ is exceptional over $X'$, then $\mult_DE\geq 0$ by Definition \ref{defn: models}(2.c).
\end{itemize}
Therefore, $p_*E\geq 0$. Since $K_{X'}+B'+\Mm_{X'}$ is nef$/U$, $q^*(K_{X'}+B'+\Mm_{X'})$ is nef$/X$, hence $E$ is anti-nef$/X$. By the negativity lemma, $E\geq 0$, which is (1).

We show (2). If $E$ is not exceptional over $X'$, then there exists a component $D$ of $E$ that is not exceptional over $X'$. If $D$ is not exceptional over $X$, then $\mult_DE=0$, a contradiction. Thus $D$ is exceptional over $X$. By the definition of weak glc models, $a(D,X',B',\Mm)=0$. Since $E\geq 0$, $a(D,X,B,\Mm)\leq a(D,X,B',\Mm)=0$. Since $(X,B,\Mm)/U$ is a glc g-pair, $a(D,X,B,\Mm)\geq 0$. Thus $a(D,X,B,\Mm)=0$, which implies that $\mult_DE=0$, a contradiction.
\end{proof}

\begin{lem}\label{lem: g-pair version bir12 2.7}
Let $(X,B,\Mm)/U$ be a glc g-pair, $(X_1,B_1,\Mm)/U$ and $(X_2,B_2,\Mm)/U$ two weak glc models of $(X,B,\Mm)/U$ with induced birational map $\phi: X_1\dashrightarrow X_2$, and $g_1: W\rightarrow X_1$ and $g_2: W\rightarrow X_2$ a common resolution such that $\phi\circ g_1=g_2$. Then:
\begin{enumerate}
    \item $$g_1^*(K_{X_1}+B_1+\Mm_{X_1})=g_2^*(K_{X_2}+B_2+\Mm_{X_2}).$$
    In particular, if $K_{X_2}+B_2+\Mm_{X_2}$ is ample$/U$, then $\phi$ is a morphism.
    \item If $K_{X_1}+B_1+\Mm_{X_1}$ is semi-ample$/U$, then for any weak glc model $(X',B',\Mm)/U$ of $(X,B,\Mm)/U$, $K_{X'}+B'+\Mm_{X'}$ is semi-ample$/U$.
\end{enumerate}
\end{lem}
\begin{proof}
Let $\phi_1: X\dashrightarrow X_1$ and $\phi_2: X\dashrightarrow X_2$ be the induced birational maps. Possibly replacing $W$, we may assume that the induced birational map $h: W\rightarrow X$ is a morphism. Let $$E_i:=h^*(K_X+B+\Mm_X)-g_i^*(K_{X_i}+B_i+\Mm_{X_i})$$
for $i\in\{1,2\}$. By Lemma \ref{lem: g-pair version bir12 2.6}, $E_i\geq 0$ and is exceptional over $X_i$ for $i\in\{1,2\}$. Thus $g_{1,*}(E_2-E_1)\geq 0$ and $E_1-E_2$ is nef$/X_1$, and $g_{2,*}(E_1-E_2)\geq 0$ and $E_2-E_1$ is nef$/X_2$. By the negativity lemma, $E_2-E_1\geq 0$ and $E_1-E_2\geq 0$. Thus $E_1=E_2$, which implies (1). (2) immediately follows from (1).
\end{proof}

\begin{lem}\label{lem: g-pair version bir12 2.8}
Let $(X,B,\Mm)/U$ be a glc g-pair, $h: W\rightarrow X$ a log resolution of $(X,\Supp B)$ such that $\Mm$ descends to $W$, and $(W,B_W,\Mm)$ a log smooth model of $(X,B,\Mm)$. Then any weak glc model (resp. log minimal model, good minimal model) of $(W,B_W,\Mm)/U$ is a weak glc model (resp. log minimal model, good minimal model) of $(X,B,\Mm)/U$. 
\end{lem}

\begin{proof}
Since $(W,B_W,\Mm)$ is a log smooth model of $(X,B,\Mm)$, we may write
$$K_W+B_W+\Mm_W=h^*(K_X+B+\Mm_X)+E$$
for some $E\geq 0$ that is $h$-exceptional.

\begin{claim}\label{claim: log smooth model log discrepancy compare}
Let $(X',B',\Mm)/U$ be a weak glc model of $(W,B_W,\Mm)/U$. Then $a(D,X,B,\Mm)\leq a(D,X',B',\Mm)$ for any prime divisor $D$ over $X$.
\end{claim}
\begin{proof}
Let $\phi_W: W\dashrightarrow X'$ be the induced birational map, and let $p: V\rightarrow W$ and $q: V\rightarrow X'$ be a common resolution such that $q=\phi_W\circ p$. By Lemma \ref{lem: g-pair version bir12 2.6},
$$p^*(K_W+B_W+\Mm_W)=q^*(K_{X'}+B'+\Mm_{X'})+F$$
for some $F\geq 0$ that is exceptional over $X'$. Then we have
$$p^*h^*(K_X+B+\Mm_X)=q^*(K_{X'}+B'+\Mm_{X'})+F-p^*E,$$
thus
$$p^*E-F\sim_{\Rr,X}q^*(K_{X'}+B'+\Mm_{X'})$$
is nef$/X$. Since $h_*p_*(F-p^*E)=h_*p_*F\geq 0$, by the negativity lemma, $F\geq p^*E$. Thus $a(D,X,B,\Mm)\leq a(D,X',B',\Mm)$ for any prime divisor $D$ over $X$ or $X'$.
\end{proof}

\noindent\textit{Proof of Lemma \ref{lem: g-pair version bir12 2.8} continued}. First we prove the weak glc model case. Let $(X',B',\Mm)/U$ be a weak glc model of $(W,B_W,\Mm)/U$ with induced birational map $\phi_W: W\dashrightarrow X'$. We check Definition \ref{defn: models}(2) for $(X,B,\Mm)/U$ and $(X',B',\Mm)/U$. Definition \ref{defn: models}(2.b) holds by construction. For any prime divisor $D$ on $X$ which is exceptional over $X'$, $h^{-1}_*D$ is a prime divisor on $W$ which is exceptional over $X'$. Thus
$$a(D,X,B,\Mm)=a(D,W,B_W,\Mm)\leq a(D,X',B',\Mm),$$
and we have Definition \ref{defn: models}(2.c). Thus we only need to show that $(X',B',\Mm)/U$ is a log birational model of $(X,B,\Mm)/U$. Let $\phi: X\dashrightarrow X'$ be the induced morphism and $B'':=\phi_*B+\Exc(\phi^{-1})$, then we only need to show that $B'=B''$. By construction, $B'=(\phi_W)_*B_W+\Exc(\phi_W^{-1})$. Let $D$ be a prime divisor on $X'$. There are three cases:

\medskip

\noindent\textbf{Case 1}. $D$ is not exceptional over $X$. In this case,
\begin{align*}
    1-\mult_DB''&=a(D,X',B'',\Mm)=a(D,X,B,\Mm)\\
    &=a(D,W,B_W,\Mm)=a(D,X',B',\Mm)=1-\mult_DB',
\end{align*}
so $\mult_DB'=\mult_DB''$.

\medskip

\noindent\textbf{Case 2}. $D$ is exceptional over $W$. In this case, $D$ is a component of $\Exc(\phi_W^{-1})$ and a component of $\Exc(\phi^{-1})$, hence
$$\mult_DB'=1=\mult_DB''.$$

\medskip

\noindent\textbf{Case 3}. $D$ is exceptional over $X$ but not exceptional over $W$. In this case,
$$1-\mult_DB'=a(D,X',B',\Mm)=a(D,W,B_W,\Mm).$$
Since $E\geq 0$,
$a(D,W,B_W,\Mm)\leq a(D,X,B,\Mm).$
By Claim \ref{claim: log smooth model log discrepancy compare}, 
$a(D,X,B,\Mm)\leq a(D,X',B',\Mm).$
Thus
$$a(D,X,B,\Mm)=a(D,X',B',\Mm)=a(D,W,B_W,\Mm).$$
By Definition \ref{defn: log smooth models}(4), $$a(D,X,B,\Mm)=a(D,X',B',\Mm)=a(D,W,B_W,\Mm)=0,$$
which implies that
$$\mult_DB'=1=\mult_D\Exc(\phi^{-1})=\mult_DB''.$$
Thus $B'=B''$, so $(X',B',\Mm)/U$ is a log birational model of $(X,B,\Mm)/U$, and we have proved the weak glc model case.

Next we prove the log minimal model case. Let $(X',B',\Mm)/U$ be a log minimal model of $(W,B_W,\Mm)/U$. We check Definition \ref{defn: models}(3) for $(X,B,\Mm)/U$ and $(X',B',\Mm)/U$. Definition \ref{defn: models}(3.a) follows from (1). Definition \ref{defn: models}(3.b) is immediate from the construction. For any prime divisor $D$ on $X$ which is exceptional over $X'$, $f^{-1}_*D$ is a prime divisor on $W$ which is exceptional over $X'$. Thus
$$a(D,X,B,\Mm)=a(D,W,B_W,\Mm)<a(D,X',B',\Mm).$$
so we get Definition \ref{defn: models}(3.c), and we have the log minimal model case.

The good minimal model case follows immediately from the log minimal model case.
\end{proof}

\subsubsection{Models under the MMP}

\begin{lem}\label{lem: run mmp keep minimal model}
Let $(X,B,\Mm)/U$ be a glc g-pair and $X\dashrightarrow Y$ a partial $(K_X+B+\Mm_X)$-MMP$/U$. Let $B_Y$ be the strict transform of $B$ on $Y$. Then any weak glc model (resp. log minimal model, good minimal model) of $(Y,B_Y,\Mm)/U$ is a weak glc model (resp. log minimal model, good minimal model) of $(X,B,\Mm)/U$.
\end{lem}
\begin{proof}
Since $X\dashrightarrow Y$ is a partial $(K_X+B+\Mm_X)$-MMP$/U$, $X\dashrightarrow Y$ does not extract any divisor, $a(D,X,B,\Mm)\leq a(D,Y,B_Y,\Mm)$ for any prime divisor $D$ over $X$ or $Y$, and $a(D,X,B,\Mm)<a(D,Y,B_Y,\Mm)$ for any prime divisor $D$ on $X$ that is exceptional over $Y$.

Assume that $(X',B',\Mm)/U$ is a weak glc model of $(Y,B_Y,\Mm)/U$, and $\phi: X\dashrightarrow X'$ and $\phi_Y: Y\dashrightarrow X'$ are the induced birational maps. Then $B'=(\phi_Y)_*B_Y+\Exc(\phi_Y^{-1})$. Let $B'':=\phi_*B+\Exc(\phi^{-1})$ and let $D$ be a prime divisor on $X'$. There are three possibilities:

\medskip

\noindent\textbf{Case 1}. $D$ is not exceptional over $X$ and $Y$. In this case,
\begin{align*}
    1-\mult_DB'&=a(D,X',B',\Mm)=a(D,Y,B_Y,\Mm)\\
    &=a(D,X,B,\Mm)=a(D,X',B'',\Mm)=1-\mult_DB'',
\end{align*}
so $\mult_DB'=\mult_DB''$.

\medskip

\noindent\textbf{Case 2}. $D$ is exceptional over $X$ and $Y$. In this case, $D$ is a component of $\Exc(\phi_Y^{-1})$ and a component of $\Exc(\phi^{-1})$, hence
$$\mult_DB'=1=\mult_DB''.$$

\medskip

\noindent\textbf{Case 3}. $D$ is exceptional over $Y$ but not exceptional over $X$. In this case, $D$ is a component of $B'$ and $a(D,X',B',\Mm)=0$. By Lemma \ref{lem: g-pair version bir12 2.6},
$$a(D,X,B,\Mm)\leq a(D,Y,B_Y,\Mm)\leq a(D,X',B',\Mm)=0,$$
so $a(D,X,B,\Mm)=0$, hence $$\mult_DB''=1-a(D,X',B'',\Mm)=1-a(D,X,B,\Mm)=1=\mult_DB'.$$

Thus $B'=B''$, hence $(X',B',\Mm)$ is a log birational model of $(X,B,\Mm)$. By Lemma \ref{lem: g-pair version bir12 2.6}(1), 
$$a(D,X,B,\Mm)\leq a(D,Y,B_Y,\Mm)\leq a(D,X',B',\Mm)$$
for any prime divisor $D$ over $X$. Since $K_{X'}+B'+\Mm_{X'}$ is nef over $U$, by Definition \ref{defn: models}(2), $(X',B',\Mm)/U$ is a weak glc model of $(X,B,\Mm)/U$, and we have proven the weak glc model case of the lemma. 

Now assume that $(X',B',\Mm)/U$ is a log minimal model of $(Y,B_Y,\Mm)/U$. By the  weak glc model case of the lemma, $(X',B',\Mm)/U$ is a weak glc model of $(X,B,\Mm)/U$. For any prime divisor $D$ on $X$ which is exceptional over $X'$, if $D$ is exceptional over $Y$, then $$a(D,X,B,\Mm)<a(D,Y,B_Y,\Mm)\leq a(D,X',B',\Mm),$$
and if $D$ is not exceptional over $Y$, then
$$a(D,X,B,\Mm)=a(D,Y,B_Y,\Mm)<a(D,X',B',\Mm).$$
Thus $(X',B',\Mm)/U$ is a log minimal model of $(X,B,\Mm)/U$ by Definition \ref{defn: models}(3), and we have proven the log minimal model case of the lemma. The good minimal model case of the lemma follows from the log minimal case.
\end{proof}

\subsubsection{Models under gdlt modifications}

\begin{lem}\label{lem: model keep under gdlt modification}
Let $(X,B,\Mm)/U$ be a glc g-pair and $(Y,B_Y,\Mm)$ a gdlt model of $(X,B,\Mm)$. Then any log birational model (resp. weak glc model, log minimal model, good minimal model) of $(Y,B_Y,\Mm)/U$ is a log birational model (resp. weak glc model, log minimal model, good minimal model) of $(X,B,\Mm)/U$.
\end{lem}
\begin{proof}
 We begin by proving the log birational model case. Let $(X',B',\Mm)/U$ be a log birational model of $(Y,B_Y,\Mm)/U$ with induced birational maps $\phi_Y: Y\dashrightarrow X'$ and $\phi: X\dashrightarrow X'$. Let $B'':=\phi_*B+\Exc(\phi^{-1})$, then for any prime divisor $D$ on $X'$, there are three cases:

\medskip

\noindent\textbf{Case 1}. $D$ is not exceptional over $X$. In this case,
\begin{align*}
    1-\mult_DB''&=a(D,X',B'',\Mm)=a(D,X,B,\Mm)\\
    &=a(D,Y,B_Y,\Mm)=a(D,X',B',\Mm)=1-\mult_DB',
\end{align*}
so $\mult_DB'=\mult_DB''$.

\medskip

\noindent\textbf{Case 2}. $D$ is exceptional over $Y$. In this case, $D$ is a component of $\Exc(\phi_Y^{-1})$ and a component of $\Exc(\phi^{-1})$, hence
$$\mult_DB'=1=\mult_DB''.$$

\medskip

\noindent\textbf{Case 3}. $D$ is exceptional over $X$ but not exceptional over $Y$. In this case, $a(D,X,B,\Mm)=a(D,Y,B_Y,\Mm)=0$. Thus
$$\mult_DB'=1-a(D,X',B',\Mm)=1-a(D,Y,B_Y,\Mm)=1=\mult_D\Exc(\phi^{-1})=\mult_DB''.$$

Thus $B'=B''$, so $(X',B',\Mm)/U$ is a log birational model of $(X,B,\Mm)/U$.

The remainder of the lemma now follows easily. In particular, notice that as  $a(D,X,B,\Mm)=a(D,Y,B_Y,\Mm)$ for any prime divisor $D$ over $X$ and $X\dashrightarrow Y$ does not contract any divisor, properties (2.c) and (3.c) of Definition \ref{defn: models} follow immediately.
\end{proof}

\subsection{Models under pullbacks}

The goal of this subsection is the following theorem, which will be proven at the end of this subsection.

\begin{thm}\label{thm: existence good minimal model under pullbacks}
Let $(X,B,\Mm)/U$ and $(Y,B_Y,\Mm)/U$ be two NQC glc g-pairs and let $f: Y\rightarrow X$ be a projective birational morphism such that
$$K_Y+B_Y+\Mm_Y=f^*(K_X+B+\Mm_X)+E$$
for some $E\geq 0$ that is exceptional over $X$. Then $(X,B,\Mm)/U$ has a weak glc model (resp. log minimal model, good minimal model) if and only if $(Y,B_Y,\Mm)/U$ has a weak glc model (resp. log minimal model, good minimal model).
\end{thm}

We prove several lemmas before proving Theorem \ref{thm: existence good minimal model under pullbacks}.

\begin{lem}\label{lem: g-pair weak glc imply lmm}
Let $(X,B,\Mm)/U$ be a glc g-pair. If $(X,B,\Mm)/U$ has a weak glc model, then $(X,B,\Mm)/U$ has a log minimal model.
\end{lem}
\begin{proof}
Let $(X',B',\Mm)/U$ be a weak glc model of $(X,B,\Mm)/U$. Let $h: W\rightarrow X$ be a log resolution of $(X,\Supp B)$ such that the induced map $\phi_W: W\rightarrow X'$ is a morphism, and $\Mm$ descends to $W$. We may write
$$K_W+B_W+\Mm_W=h^*(K_X+B+\Mm_X)+E$$
for some log smooth pair $(W,B_W)$, such that $B_W:=h^{-1}_*B+\Exc(h)$ and $E\geq 0$ is exceptional over $X$. Then $(W,B_W,\Mm)$ is a log smooth model of $(X,B,\Mm)$. By Lemma \ref{lem: g-pair version bir12 2.6}, we have
$$h^*(K_X+B+\Mm_X)=\phi_W^*(K_{X'}+B'+\Mm_{X'})+G$$
where $G\geq 0$ is exceptional over $X'$. Thus
$$K_W+B_W+\Mm_W\sim_{\mathbb R,X'}G+E.$$

\begin{claim}\label{claim: wglc to lmm E exceptional}
$E$ is exceptional over $X'$.
\end{claim}
\begin{proof}
Let $D$ be a component of $E$. By construction, $a(D,X,B,\Mm)>0$ and $D$ is exceptional over $X$. 

Assume that  $D$ is not exceptional over $X'$. Since $(X',B',\Mm)/U$ is a log birational model of $(X,B,\Mm)/U$, $a(D,X',B',\Mm)=0$. Since $G\geq 0$, $a(D,X,B,\Mm)\leq a(D,X',B',\Mm)$. Thus $a(D,X,B,\Mm)=0$, hence $D$ is not a component of $E$, a contradiction.
\end{proof}

\noindent\textit{Proof of Lemma \ref{lem: g-pair weak glc imply lmm} continued}. By Claim \ref{claim: wglc to lmm E exceptional}, $G+E$ is exceptional over $X'$. By Lemma \ref{lem: rlinear version of hl18 3.8}, we may run a $(K_W+B_W+\Mm_W)$-MMP$/X'$ with scaling of a general ample$/X'$ divisor, which terminates with a model $Y$ such that $K_Y+B_Y+\Mm_Y\sim_{\Rr,X'}0$, where $B_Y$ is the strict transform of $B$ on $Y$. Applying the negativity lemma twice, we have that $K_Y+B_Y+\Mm_Y$ is the pullback of $K_{X'}+B'+\Mm_{X'}$. Thus $K_Y+B_Y+\Mm_Y$ is nef$/U$. Since $(W,B_W,\Mm)$ is $\Qq$-factorial gdlt and $W\dashrightarrow Y$ is a $(K_W+B_W+\Mm_W)$-MMP$/X'$, $(Y,B_Y,\Mm)$ is $\Qq$-factorial gdlt. Thus $(Y,B_Y,\Mm)/U$ is a log minimal model of $(W,B_W,\Mm)/U$. The lemma follows from Lemma \ref{lem: g-pair version bir12 2.8}.
\end{proof}

\begin{lem}\label{lem: same weak glc model under pullback}
Let $(X,B,\Mm)/U$ and $(Y,B_Y,\Mm)/U$ be two glc g-pairs, and $f: Y\rightarrow X$ a projective birational morphism such that
$$K_Y+B_Y+\Mm_Y=f^*(K_X+B+\Mm_X)+E$$
for some $E\geq 0$ that is exceptional over $X$. Then any weak glc model of $(X,B,\Mm)/U$ is a weak glc model of $(Y,B_Y,\Mm)/U$.
\end{lem}
\begin{proof}
Let $(X',B',\Mm)/U$ be a weak glc model of $(X,B,\Mm)/U$, $\phi: X\dashrightarrow X'$ the induced birational map, and $\phi_Y:=\phi\circ f$. Let $p: W\rightarrow Y$ and $q: W\rightarrow X'$ be a common resolution and let $h:=f\circ p$. 
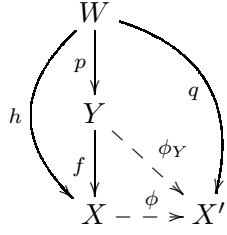
\begin{center}$\xymatrix{
W\ar@{->}[d]_{p}\ar@/^2pc/[ddr]_q\ar@/_2pc/[dd]_h  \\
Y\ar@{->}[d]_{f}\ar@{-->}[dr]^{\phi_Y}&    \\
 X\ar@{-->}[r]^{\phi}& X'
}$
\end{center}
By Lemma \ref{lem: g-pair version bir12 2.6},
$$h^*(K_X+B+\Mm_X)=q^*(K_{X'}+B'+\Mm_{X'})+F$$
for some $F\geq 0$ that is exceptional over $X'$. Thus 
$$p^*(K_Y+B_Y+\Mm_Y)=q^*(K_{X'}+B'+\Mm_{X'})+p^*E+F.$$
Thus $a(D,Y,B_Y,\Mm)\leq a(D,X',B',\Mm)$ for any prime divisor $D$ over $X'$. In particular, if $a(D,X',B',\Mm)=0$, then $a(D,Y,B_Y,\Mm)=0$.

Since $(X',B',\Mm)/U$ is a log birational model of $(X,B,\Mm)/U$, $B'=\phi_*B+\Exc(\phi^{-1})$. Let $B'':=(\phi_Y)_*B_Y+\Exc(\phi_Y^{-1})$. For any prime divisor $D$ on $X'$, there are two cases:

\medskip

\noindent\textbf{Case 1}. $D$ is not exceptional over $X$. In this case,
\begin{align*}
   1-\mult_DB'&=a(D,X',B',\Mm)=a(D,X,B,\Mm)\\
   &=a(D,Y,B_Y,\Mm)=a(D,X',B'',\Mm)=1-\mult_DB'',
\end{align*}
so $\mult_DB'=\mult_DB''$.

\medskip

\noindent\textbf{Case 2}. $D$ is exceptional over $X$. In this case, 
$$a(D,X',B',\Mm)=1-\mult_DB'=0.$$
Since $a(D,Y,B_Y,\Mm)\leq a(D,X',B',\Mm)$, $a(D,Y,B_Y,\Mm)=0$. Thus if $D$ is not exceptional over $Y$, then
$$\mult_DB''=\mult_DB_Y=1-a(D,Y,B_Y,\Mm)=1=\mult_DB',$$
and if $D$ is exceptional over $Y$, then
$$\mult_DB''=\mult_D\Exc(\phi_Y^{-1})=1=\mult_DB'.$$

Thus $B'=B''$, hence $(X',B',\Mm)/U$ is a log birational model of $(Y,B_Y,\Mm)/U$. Since $K_{X'}+B'+\Mm_{X'}$ is nef$/U$, and
$a(D,Y,B_Y,\Mm)\leq a(D,X',B',\Mm)$ for any prime divisor $D$ over $X'$, $(X',B',\Mm)/U$ is a weak glc model of $(Y,B_Y,\Mm)/U$.
\end{proof}

\begin{lem}\label{lem: existence good minimal model under pullback}
Let $(X,B,\Mm)/U$ and $(Y,B_Y,\Mm)/U$ be two glc g-pairs, and $f: Y\rightarrow X$ a projective birational morphism such that
$$K_Y+B_Y+\Mm_Y=f^*(K_X+B+\Mm_X)+E$$
for some $E\geq 0$ that is exceptional over $X$. If $(X,B,\Mm)/U$ has a weak glc model (resp. log minimal model, good minimal model), then $(Y,B_Y,\Mm)/U$ has a weak glc model (resp. log minimal model, good minimal model).
\end{lem}
\begin{proof}
Let $(X',B',\Mm)/U$ be a weak glc model (resp. log minimal model, good minimal model) of $(X,B,\Mm)/U$. By Lemma \ref{lem: same weak glc model under pullback}, $(X',B',\Mm)/U$ is a weak glc model of $(Y,B_Y,\Mm)/U$. By Lemma \ref{lem: g-pair weak glc imply lmm}, $(Y,B_Y,\Mm)/U$ has a log minimal model $(Y',B_{Y'},\Mm)/U$. By Lemma \ref{lem: g-pair version bir12 2.7}, if $K_{X'}+B'+\Mm_{X'}$ is semi-ample$/U$, then $K_{Y'}+B_{Y'}+\Mm_{Y'}$ is semi-ample$/U$, and we finish the proof.
\end{proof}

\begin{lem}\label{lem: existence good minimal model under pullbacks weak glc case}
Let $(X,B,\Mm)/U$ and $(Y,B_Y,\Mm)/U$ be two $\Qq$-factorial NQC gdlt g-pairs, and $f: Y\rightarrow X$ a projective birational morphism such that
$$K_Y+B_Y+\Mm_Y=f^*(K_X+B+\Mm_X)+E$$
for some $E\geq 0$ that is exceptional over $X$. Assume that
\begin{enumerate}
\item $\Mm$ descends to $Y$,
\item $(Y,B_Y+\Exc(f))$ is log smooth, and
\item $(Y,B_Y,\Mm)/U$ has a weak glc model.
\end{enumerate}
Then $(X,B,\Mm)/U$ has a weak glc model.
\end{lem}
\begin{proof}
By our assumption, $K_Y+B_Y+\Mm_Y$ and $K_X+B+\Mm_X$ are pseudo-effective$/U$. Since $(X,B,\Mm)$ is gdlt, we may pick an ample$/U$ $\Rr$-divisor $A\geq 0$ on $X$ such that $(X,B+A,\Mm)$ is glc, and $K_X+B+A+\Mm_X$ and $A+\lfloor B\rfloor$ are ample. Since $(X,B,\Mm)$ is gdlt, $(X,\{B\},\Mm)$ is gklt, so we may pick an ample$/U$ $\Rr$-divisor $0\leq A'\sim_{\mathbb R,U}A+\lfloor B\rfloor$ such that $(X,\Delta:=\{B\}+A',\Mm)$ is gklt and $f$ is a log resolution of $(X,B+A)$. Since $\Delta\sim_{\Rr,U}B+A$, $K_X+\Delta+\Mm_X$ is big$/U$. We may write
$$K_Y+\Gamma+\Mm_Y=f^*(K_X+\Delta+\Mm_X)+F$$
for some $\Gamma\geq 0$, $F\geq 0$ such that $\Gamma\wedge F=0$. By our construction, $\Mm$ descends to $Y$, $(Y,B_Y+\Exc(f))$ is log smooth, $(Y,\Gamma)$ is log smooth, $(Y,\Gamma,\Mm)$ is gklt, and $K_Y+\Gamma+\Mm_Y$ is big$/U$. We let
$$\Delta_t:=t\Delta+(1-t)B\sim_{\mathbb R,U}B+tA$$
and
$$\Gamma_t:=t\Gamma+(1-t)B_Y$$
for any real number $t$. Then $(X,\Delta_t,\Mm)$ and $(Y,\Gamma_t,\Mm)$ are gklt for any $t\in (0,1]$, and $K_X+\Delta_t+\Mm_X$ and $K_Y+\Gamma_t+\Mm_Y$ are big$/U$ for any $t\in (0,1]$.

Since $(Y,B_Y,\Mm)/U$ has a weak glc model, by Lemma \ref{lem: g-pair weak glc imply lmm},  $(Y,B_Y,\Mm)/U$ has a log minimal model. Since $Y$ is klt, by Theorem \ref{thm: mmp with scaling gpair terminates assuming gmm}, we may run a $(K_Y+B_Y+\Mm_Y)$-MMP$/U$ with scaling of a general ample$/U$ divisor $H$, which terminates with a log minimal model $(Y',B_{Y'},\Mm)/U$ with induced birational map $\phi: Y\dashrightarrow Y'$. 

We let $\Gamma'_t$ be the strict transform of $\Gamma_t$ on $Y'$ for any $t$. By Lemmas \ref{lem: still an mmp under perturbation} and \ref{lem: trivial mmp under perturbation}, there exists $t_0\in (0,1)$, such that 
\begin{itemize}
    \item $\phi$ is also a $(K_Y+\Gamma_{t_0}+\Mm_Y)$-MMP$/U$, and
    \item for any $t\in (0,t_0]$, any partial $(K_{Y'}+\Gamma'_{t}+\Mm_{Y'})$-MMP$/U$ is $(K_{Y'}+B_{Y'}+\Mm_{Y'})$-trivial.
\end{itemize}
Thus $(Y',\Gamma_{t_0}',\Mm)$ is gklt and $K_{Y'}+\Gamma'_{t_0}+\Mm_{Y'}$ is big$/U$. By Theorem \ref{thm: bz16 4.4(2)}, we may run a $(K_{Y'}+\Gamma'_{t_0}+\Mm_{Y'})$-MMP$/U$, which terminates with a log minimal model $(Y'',\Gamma''_{t_0},\Mm)/U$ of $(Y',\Gamma'_{t_0},\Mm)/U$. We let $\Gamma''_t$ be the strict transform of $\Gamma_t$ on $Y''$ for any $t$ and $B_{Y''}$ the strict transform of $B_Y$ on $Y''$. Since the induced birational map $\phi': Y'\dashrightarrow Y''$ is $(K_{Y'}+B_{Y'}+\Mm_{Y'})$-trivial, $K_{Y''}+B_{Y''}+\Mm_{Y''}$ is nef$/U$. Moreover, the induced map $\phi'\circ\phi: Y\dashrightarrow Y''$ does not extract any divisor, and is both $(K_Y+B_Y+\Mm_Y)$-non-positive and $(K_Y+\Gamma_{t_0}+\Mm_Y)$-non-positive. Thus $(Y'',B_{Y''},\Mm)/U$ is a weak glc model of $(Y,B_Y,\Mm)/U$ and $(Y'',\Gamma''_{t_0},\Mm)/U$ is a weak glc model of $(Y,\Gamma_{t_0},\Mm)/U$, hence $(Y'',\Gamma''_t,\Mm)/U$ is a weak glc model of $(Y,\Gamma_t,\Mm)/U$ for any $t\in [0,t_0]$.

By Theorem \ref{thm: can run mmp for gklt pair}, we can run a $(K_X+B+\Mm_X)$-MMP$/U$ with scaling of $A$:
$$(X,B,\Mm):=(X_1,B_1,\Mm)\dashrightarrow (X_2,B_2,\Mm)\dashrightarrow\dots\dashrightarrow (X_i,B_i,\Mm)\dashrightarrow\dots.$$
Let $A_i,\Delta_i,\Delta_{t,i}$ be the strict transforms of $A,\Delta,\Delta_t$ on $X_i$ for any $t,i$ respectively, and let
$$\lambda_i:=\inf\{t\mid t\geq 0, K_{X_i}+B_i+tA_i+\Mm_{X_i}\text{ is nef/}U\}$$
be the scaling numbers. If this MMP terminates, then there is nothing left to prove as we already get a log minimal model for $(X,B,\Mm)/U$. Thus we may assume that this MMP does not terminate. By Theorem \ref{thm: mmp with scaling gpair terminates assuming gmm}, $\lim_{i\rightarrow+\infty}\lambda_i=0$. 

In particular, there exists a positive integer $n$ such that $\lambda_n<\lambda_{n-1}\leq t_0$. Since $\Delta_{t,i}\sim_{\mathbb R,U}B_i+tA_i$ for any $t$, $(X_n,\Delta_{\lambda_{n-1},n},\Mm)/U$ is a weak glc model of $(X,\Delta_{\lambda_{n-1}},\Mm)/U$ and $(X_n,\Delta_{\lambda_{n},n},\Mm)/U$ is a weak glc model of $(X,\Delta_{\lambda_{n}},\Mm)/U$. Since
$$K_Y+\Gamma_t+\Mm_Y=f^*(K_X+\Delta_t+\Mm_X)+tF+(1-t)E$$
for any $t$, by Lemma \ref{lem: same weak glc model under pullback}, $(X_n,\Delta_{\lambda_{n-1},n},\Mm)/U$ is a weak glc model of $(Y,\Gamma_{\lambda_{n-1}},\Mm)/U$ and $(X_n,\Delta_{\lambda_{n},n},\Mm)/U$ is a weak glc model of $(Y,\Gamma_{\lambda_{n}},\Mm)/U$. By our construction, $(Y'',\Gamma''_{\lambda_{n-1}},\Mm)/U$ is a weak glc model of $(Y,\Gamma_{\lambda_{n-1}},\Mm)/U$ and $(Y'',\Gamma''_{\lambda_{n}},\Mm)/U$ is a weak glc model of $(Y,\Gamma_{\lambda_{n}},\Mm)/U$.

We let $p: W\rightarrow X_n$ and $q: W\rightarrow Y''$ be a resolution of indeterminacy. 
\begin{center}$\xymatrix{
Y\ar@{->}[d]_{f}\ar@{-->}[r]^{\phi}& Y'\ar@{-->}[r]^{\phi'} & Y''& W\ar@{->}[d]^{p}\ar@{->}[l]_{q}\\
 X\ar@{-->}[r]& X_2\ar@{-->}[r] & \dots\ar@{-->}[r] & X_n
}$
\end{center}
By Lemma \ref{lem: g-pair version bir12 2.7}(1),
$$p^*(K_{X_n}+\Delta_{\lambda_{n-1},n}+\Mm_{X_n})=q^*(K_{Y''}+\Gamma''_{\lambda_{n-1}}+\Mm_{Y''}).$$
and
$$p^*(K_{X_n}+\Delta_{\lambda_{n},n}+\Mm_{X_n})=q^*(K_{Y''}+\Gamma''_{\lambda_{n}}+\Mm_{Y''}).$$
Thus
$$p^*(K_{X_n}+t\Delta_i+(1-t)B_i+\Mm_{X_n})=q^*(K_{Y''}+t\Gamma''_1+(1-t)B_{Y''}+\Mm_{Y''})$$
when $t\in\{\lambda_{n-1},\lambda_{n}\}$. Since $\lambda_{n-1}\not=\lambda_n$, we have
$$p^*(K_{X_n}+t\Delta_i+(1-t)B_i+\Mm_{X_n})=q^*(K_{Y''}+t\Gamma''_1+(1-t)B_{Y''}+\Mm_{Y''})$$
for any $t$. In particular,
$$p^*(K_{X_n}+B_n+\Mm_{X_n})=q^*(K_{Y''}+B_{Y''}+\Mm_{Y''})$$
is nef$/U$, hence $K_{X_n}+B_n+\Mm_{X_n}$ is nef$/U$, and $\lambda_n=0$, a contradiction.
\end{proof}

\begin{proof}[Proof of Theorem \ref{thm: existence good minimal model under pullbacks}]
First we prove the weak glc model case. By Lemma \ref{lem: same weak glc model under pullback}, we only need to prove that if $(Y,B_Y,\Mm)/U$ has a weak glc model, then $(X,B,\Mm)/U$ has a weak glc model. Let $g: \bar X\rightarrow X$ be a gdlt modification of $(X,B,\Mm)$ such that
$$K_{\bar X}+\bar B+\Mm_{\bar X}=g^*(K_X+B+\Mm_X),$$
and let $p: W\rightarrow Y$ and $q: W\rightarrow\bar X$ be a resolution of indeterminacy, such that $\Mm$ descends to $W$, $p$ is a log resolution of $(Y,\Supp(B_Y+E))$, and $q$ is a log resolution of $(\bar X,\Supp\bar B)$.  By Lemma \ref{lem: existence of proper log smooth model}, we may find a proper log smooth model $(W,B_W,\Mm)$ of $(Y,B_Y,\Mm)$. We have
$$K_W+B_W+\Mm_W=p^*(K_Y+B_Y+\Mm_Y)+F=(p\circ f)^*(K_X+B+\Mm_X)+p^*E+F$$
for some $p$-exceptional $\Rr$-divisor $F\geq 0$.

Let $D$ be a component of $p^*E+F$. Then $a(D,W,B_W,\Mm)<a(D,X,B,\Mm)$ and $D$ is exceptional over $X$. If $D$ is not exceptional over $\bar X$, then $a(D,W,B_W,\Mm)<a(D,X,B,\Mm)=0$, which is not possible. Thus $p^*E+F$ is exceptional over $\bar X$. 

By Lemma \ref{lem: same weak glc model under pullback}, $(W,B_W,\Mm)/U$ has a weak glc model. Since $p^*E+F$ is exceptional over $\bar X$, $\Mm$ descends to $W$, $(W,B_W+p^*E+F)$ is log smooth, by Lemma \ref{lem: existence good minimal model under pullbacks weak glc case}, we have that $(\bar X,\bar B,\Mm)/U$ has a weak glc model. By Lemma \ref{lem: model keep under gdlt modification}, $(X,B,\Mm)/U$ has a weak glc model, and we have proven the weak glc model case.

Now we prove the general case. By Lemma \ref{lem: existence good minimal model under pullback}, we only need to prove that if $(Y,B_Y,\Mm)/U$ has a weak glc (resp. log minimal model, good minimal model), then $(X,B,\Mm)/U$ has a weak glc (resp. log minimal model, good minimal model). The weak glc case has just been proven, and the log minimal model case follows from the weak glc model case and Lemma \ref{lem: g-pair weak glc imply lmm}. Assume that $(Y,B_Y,\Mm)/U$ has a good minimal model. By the log minimal model case, we may assume that $(X',B',\Mm)/U$ is a log minimal model of $(X,B,\Mm)/U$. By Lemma \ref{lem: same weak glc model under pullback}, $(X',B',\Mm)/U$ is also a weak glc model of $(Y,B_Y,\Mm)/U$. By Lemma \ref{lem: g-pair version bir12 2.7}(2), $K_{X'}+B'+\Mm_{X'}$ is semi-ample$/U$, hence $(X',B',\Mm)/U$ is a  good minimal model of $(X,B,\Mm)/U$, and the proof is concluded.
\end{proof}

\begin{cor}\label{cor: keep model under mmp}
Let $(X,B,\Mm)/U$ be an NQC glc g-pair and $X\dashrightarrow X'$ a partial $(K_X+B+\Mm_X)$-MMP$/U$. Let $B'$ be the strict transform of $B$ on $X'$. Then $(X,B,\Mm)/U$ has a weak glc model (resp. log minimal model, good minimal model) if and only if $(X',B',\Mm)/U$ has a weak glc model (resp. log minimal model, good minimal model).
\end{cor}
\begin{proof}
Let $p: W\rightarrow X$ and $q: W\rightarrow X'$ be a resolution of indeterminacy of $X\dashrightarrow X'$ such that $\Mm$ descends to $W$. Let $(W,B_W,\Mm)$ be a log smooth model of $(X,B,\Mm)$, then $(W,B_W,\Mm)$ is also a log smooth model of $(X',B',\Mm)$. By Theorem \ref{thm: existence good minimal model under pullbacks}, $(X,B,\Mm)/U$ has a weak glc model (resp. log minimal model, good minimal model) if and only if $(W,B_W,\Mm)/U$ has a weak glc model (resp. log minimal model, good minimal model) if and only if $(X',B',\Mm)/U$ has a weak glc model (resp. log minimal model, good minimal model).
\end{proof}

\section{Proof of Theorem \ref{thm: existence of glc closure} }

\begin{thm}\label{thm: reduce special gpair to pair}
Let Let $(X,B,\Mm)/U$ be a $\Qq$-factorial NQC gdlt g-pair. Assume that there exists a non-empty open subset $U^0\subset U$, such that
\begin{enumerate}
    \item the image of any strata of $S:=\lfloor B\rfloor$ in $U$ intersects $U^0$, and
    \item $\Mm^0:=\Mm\times_UU^0$ descends to $X^0:=X\times_UU^0$ and $\Mm^0_{X^0}\sim_{\Rr,U^0}0$.
\end{enumerate}
Then there exists an $\Rr$-divisor $G\sim _{\Rr,U}\Mm _X$ such that $(X,B+G)$ is lc and $\Nklt(X,B+G)=\Ngklt(X,B,\Mm)$.
\end{thm}
\begin{proof}
By Theorem \ref{thm: shokurov polytope gpair}, we may assume that $(X,B,\Mm)$ is a $\Qq$-g-pair. Possibly shrinking $U^0$, we may assume that $U^0$ is affine.

 By \cite[Proposition 6-1-3, Remark 6-1-4]{KMM87} (see also \cite[Lemma 6]{Nak86}) and Theorem \ref{thm: has19 weak semistable reduction}, we may let $f: X'\to X$ be a resolution with morphisms $\pi': X'\rightarrow U'$ and $\varphi: U'\rightarrow U$, such that
\begin{itemize}
\item $\Mm$ descends to $X'$,
\item we may write
$$K_{X'}+B_{X'}+\Mm_{X'}=f^*(K_X+B+\Mm_X)+E_{X'},$$
where $B_{X'},E_{X'}\geq 0$, $B_{X'}\wedge E_{X'}=0$, $(X',{\rm Supp}(B_{X'}+E_{X'}))$ is quasi-smooth, and
    \item $p:=\pi\circ f=\varphi\circ\pi': X'\to U$ where  $U'$ is smooth, $\pi'$ and $\varphi$ are projective, $f$ is birational, and $\pi'$ has connected equidimensional fibers.
\end{itemize}
\begin{center}$\xymatrix{
X'\ar@{->}[r]^{\pi'}\ar@{->}[d]_{f}\ar@{->}[dr]^{p} & U'\ar@{->}[d]^{\varphi}\\
X\ar@{->}[r]^{\pi} & U
}$
\end{center}
We show that there is a $\varphi$-nef $\mathbb Q$-divisor $M_{U'}$ on $U'$ such that $\Mm_{X'}\sim_{\Qq,U}\pi'^*M_{U'}$. By our construction, $\Mm_{X'}|_{X'_{\eta}}\sim_{\Qq}0$ where $X'_{\eta}$ is the generic fiber of $p$. Thus $\Mm_{X'}\sim_{\Qq}0$ over the generic point $\eta_{U'}$ of $U'$. By Lemma \ref{lem: lift equivalence from generic fiber}, $\Mm_{X'}\sim_{\Qq,U'} D$ where $D\geq 0$ is vertical over $U'$. Since $\pi'$ is equidimensional, $\pi'(D)$ is a $\Qq$-divisor on $U'$. Since $U'$ is smooth, for any prime divisor $P$ on $U'$, we may define
$$\nu_{P}:=\sup\{\nu\mid \nu\geq 0, D-\nu\pi'^*P\geq 0\},$$
then $\nu_P>0$ for only finitely many prime divisors $P$ on $U'$. Let $D':=D-\pi'^*(\sum_P\nu_{P}P)$, then $\Mm_{X'}\sim_{\Qq,U'} D'\geq 0$ and $D'$ is very exceptional over $U$. By the general negativity lemma \cite[Lemma 3.3]{Bir12a}, $\Mm_{X'}\sim_{\Qq,U'}0$. In particular, since $\Mm_{X'}$ is nef$/U$, $\Mm_{X'}\sim_{\Qq,U}\pi'^*M_{U'}$ for some $\Qq$-divisor $M_{U'}$ that is nef$/U$.

Let $X'^0:=X'\times_UU^0$ and $U'^0:=U'\times_UU^0$. Since $\Mm_{X'}|_{X'^0}\sim _{\Qq,U^0}0$, we have that $M_{U'^0}:=M_{U'}|_{U'^0}\sim_{\Qq,U^0}0$.

To prove the claim it suffices to show that for a general element $G'\in |\Mm _{X'}/U|_\Qq$, the pair 
$(X',B_{X'}+G')$ is lc and its lc centers coincide with the lc centers of $(X',B_{X'})$, i.e. the strata of $\lfloor B_{X'}\rfloor$.  If this is the case, then $(X',B_{X'}-E_{X'}+G')$ is sub-lc and $K_{X'}+B_{X'}-E_{X'}+G'\sim _\Qq f^*(K_X+B+G)$ where $G=f_* G'\in |\Mm _{X}/U|_\Qq$ and $(X,B+G)$ is log canonical and its log canonical places coincide with the glc places of $(X,B,\Mm)$.

Let $E\geq 0$ be an effective divisor on $U'$ such that $-E$ is ample over $U$ (note that $E$ is not necessarily exceptional, but its support can be chosen to avoid any point not in the exceptional locus). It follows that $|M_{U'}/U|_\Qq\supset |M_{U'}-\epsilon E/U|_\Qq+\epsilon E$.
Since $M_{U'}-\epsilon E$ is ample over $U$, for a general element $G'\in |\Mm _{X'}/U|_\Qq$ 
we have that 
the set of nklt places of $(X',B_{X'}+G')$ are contained in the set of nklt places of $(X',B_{X'}+\epsilon \pi '^*E)$. Thus, the only non-klt centers of $(X',B_{X'}+G')$ are strata of $\lfloor B_{X'} \rfloor$. 

To prove the claim, it suffices to show that the support of a general element $G'\in |\Mm _{X'}/U|_\Qq$ does not contain any stratum $S'$ of $\lfloor B_{X'} \rfloor$ or equivalently that there exist one  element $G'\in |\Mm _{X'}/U|_\Qq$ whose support does not contain any given stratum $S'$ of $\lfloor B_{X'} \rfloor$. Note that $f(S')$ is a glc center of $(X,B,\Mm)$. As $(X,B,\Mm)$ is gdlt, its glc centers are the strata of $\lfloor B\rfloor$ which intersect $X^0$ by assumption. Pick a point $x\in f(S')\cap X^0$ and let $u=\pi (x)\in U^0$. Since $M_{U'^0}\sim _{\Qq ,U^0}0$, we have $$M_{U'^0}=\alpha(h)+(\varphi|_{U^0})^*(H_0),$$
where $\alpha\in\mathbb Q$, $h$ is a rational function on $U$, and $H_0$ is a $\Qq$-divisor on $U^0$. Since $U^0$ is affine, possibly replacing $H_0$ and $\alpha(h)$, we may assume that $H_0\geq 0$ and $u\not\in\Supp H_0$. Now we may pick a sufficiently ample divisor $H\geq 0$ on $U$, such that $H|_{U^0}\geq H_0$ and $u\not\in\Supp H$, and hence $M_{U'}\sim_{\Qq,U}F:=\varphi^*H$. Then $F\geq 0$ is a $\Qq$-divisor whose support does not intersect the fiber $\varphi^{-1}(u)$, so $\pi '^*F \in |\Mm_{X'}/U|_\Qq$ and its support does not contain $S'$.
\end{proof}

\begin{lem}\label{lem: rlin equivalent good minimal model}
Let $(X,B,\Mm)/U$ and $(X,B',\Mm')/U$ be two NQC glc g-pairs, $f: Y\rightarrow X$ a birational morphism, $K_Y+B_Y+\Mm_Y:=f^*(K_X+B+\Mm_X)$ and $K_Y+B_Y'+\Mm'_Y:=f^*(K_X+B'+\Mm'_X)$, such that $Y$ is $\Qq$-factorial klt and $(Y,B_Y,\Mm)/U$ and $(Y,B_Y',\Mm')/U$ are glc g-pairs.

Assume that there exists a positive real number $r$ such that $K_X+B+\Mm_X\sim_{\Rr,U}r(K_X+B'+\Mm'_X)$. Then $(X,B,\Mm)/U$ has a good minimal model if and only if $(X,B',\Mm')$ has a good minimal model.
\end{lem}
\begin{proof}
Let $A_Y$ be a general ample$/U$ divisor on $Y$ such that $(Y,B_Y+A_Y,\Mm)/U$ and $(Y,B'_Y+A_Y,\Mm')/U$ are glc, and $K_Y+B_Y+A_Y+\Mm_Y$ and $K_Y+B'_Y+rA_Y+\Mm'_Y$ are nef$/U$.

Without lost of generality, we may assume that $(X,B,\Mm)/U$ has a good minimal model and only need to show that $(X,B',\Mm')/U$ has a good minimal model. By Theorem \ref{thm: existence good minimal model under pullbacks}, $(Y,B_Y,\Mm)/U$ has a good minimal model. By Theorem \ref{thm: mmp with scaling gpair terminates assuming gmm} and Lemma \ref{lem: g-pair version bir12 2.7}(2), we may let $\phi: Y\dashrightarrow Z$ be a $(K_Y+B_Y+\Mm_Y)$-MMP$/U$ with scaling of $A_Y$, such that $(Z,B_Z,\Mm)/U$ is a weak glc model of $(Y,B_Y,\Mm)/U$ and $K_Z+B_Z+\Mm_Z$ is semi-ample$/U$, where $B_Z$ is the strict transform of $B_Y$ on $Z$. Then $\phi$ is also a $(K_Y+B'_Y+\Mm'_Y)$-MMP$/U$ with scaling of $rA_Y$. We let $B_Z'$ be the strict transform of $B'_Y$ on $Z$, then $K_Z+B_Z+\Mm_Z\sim_{\Rr,U}r(K_Z+B'_Z+\Mm'_Z)$. Thus $(Z,B'_Z,\Mm')/U$ is a weak glc model of $(Y,B_Y',\Mm')/U$ and $K_Z+B'_Z+\Mm'_Z$ is semi-ample$/U$. By Lemmas \ref{lem: g-pair weak glc imply lmm} and \ref{lem: g-pair version bir12 2.7}(2), $(Y,B_Y',\Mm')$ has a good minimal model. By Theorem \ref{thm: existence good minimal model under pullbacks}, $(X,B',\Mm')/U$ has a good minimal model.
\end{proof}

\begin{proof}[Proof of Theorem \ref{thm: existence of glc closure}]
By Definition-Lemma \ref{deflem: gdlt modification} and Theorem \ref{thm: existence good minimal model under pullbacks}, possibly replacing $(X,B,\Mm)$ with a gdlt modification, we may assume that $(X,B,\Mm)$ is $\Qq$-factorial gdlt. By Theorem \ref{thm: reduce special gpair to pair}, we may find an $\Rr$-divisor $0\leq G\sim_{\mathbb R}\Mm_X$ such that $(X,B+G)$ is lc and $\Nklt(X,B+G)=\Ngklt(X,B,\Mm)$. By \cite[Theorem 1.2]{Has19} (see also \cite[Theorem 1.1]{HX13}), $(X,B+G)/U$ has a good minimal model. By Lemma \ref{lem: rlin equivalent good minimal model}, $(X,B,\Mm)/U$ has a good minimal model.
\end{proof}

\section{Base-point-free, contraction, and cone theorems for generalized pairs}

In this section, we prove Theorem \ref{thm: cone and contraction theorem glc pair}. For the reader's convenience, we will prove Theorem \ref{thm: cone and contraction theorem glc pair}(1)(2)(3) (the cone theorem) and Theorem \ref{thm: cone and contraction theorem glc pair}(4) (the contraction theorem) separately, and we will also prove a base-point-free theorem. More precisely, we will prove the following three theorems:

\begin{thm}[Cone theorem for glc g-pairs]\label{thm: cone theorem glc g-pairs}
Let $(X,B,\Mm)/U$ be an NQC glc g-pair and $\pi: X\rightarrow U$ the associated projective morphism. Let $\{R_j\}_{j\in\Lambda}$ be the set of $(K_X+B+\Mm_X)$-negative extremal rays in $\overline{NE}(X/U)$ that are rational. Then:
\begin{enumerate}
    \item $$\overline{NE}(X/U)=\overline{NE}(X/U)_{K_X+B+\Mm_X\geq 0}+\sum_{j\in\Lambda}R_j.$$
    In particular, any $(K_X+B+\Mm_X)$-negative extremal ray in $\overline{NE}(X/U)$ is rational.
    \item Each $R_j$ is spanned by a rational curve $C_j$ such that $\pi(C_j)=\{pt\}$ and $$0<-(K_X+B+\Mm_X)\cdot C_j\leq 2\dim X.$$
    \item For any ample$/U$ $\Rr$-divisor $A$ on $X$,
    $$\Lambda_A:=\{j\in\Lambda\mid R_j\subset\overline{NE}(X/U)_{K_X+B+\Mm_X+A<0}\}$$
    is a finite set. In particular, $\{R_j\}_{j\in\Lambda}$ is countable, and is a discrete subset in $\overline{NE}(X/U)_{K_X+B+\Mm_X+A<0}$. Moreover, we may write
    $$\overline{NE}(X/U)=\overline{NE}(X/U)_{K_X+B+\Mm_X+A\geq 0}+\sum_{j\in\Lambda_A}R_j.$$
    \item Let $F$ be a $(K_X+B+\Mm_X)$-negative extremal face in $\overline{NE}(X/U)$. Then $F$ is a rational extremal face.
\end{enumerate}
\end{thm}

\begin{thm}[Base-point-free theorem for glc g-pairs]\label{thm: base-point-free theorem for glc pairs}
Let $(X,B,\Mm)/U$ be an NQC glc g-pair and $\pi: X\rightarrow U$ the associated projective morphism. Assume that $\Mm_X$ is $\Rr$-Cartier. Let $L$ be a $\pi$-nef Cartier divisor on $X$ that is the supporting function of a $(K_X+B+\Mm_X)$-negative extremal ray. Then $mL$ is $\pi$-generated for any integer $m\gg 0$.
\end{thm}

\begin{thm}[Contraction theorem for glc g-pairs]\label{thm: contraction theorem glc g-pairs}
Let $(X,B,\Mm)/U$ be an NQC glc g-pair and $\pi: X\rightarrow U$ the associated projective morphism. Assume that $\Mm_X$ is $\Rr$-Cartier. Let $R$ be a $(K_X+B+\Mm_X)$-negative extremal ray. Then there exists a projective morphism $\cont_R: X\rightarrow Y$ over $U$ satisfying the following:
\begin{enumerate}
    \item Let $C$ be an integral curve such that $\pi(C)$ is a point. Then $\cont_R(C)$ is a point if and only if $[C]\in R$.
    \item $\mathcal{O}_Y\cong(\cont_R)_*\mathcal{O}_X$. In other words, $\cont_R$ is a contraction.
    \item Let $L$ be a line bundle on $X$ such that $L\cdot C=0$ for every curve $C$ such that $[C]\in R$. Then there exists a line bundle $L_Y$ on $Y$ such that $L\cong f^*L_Y$.
\end{enumerate}
\end{thm}

\subsection{Preliminary results on non-lc pairs}

We begin by recalling some results on non-lc pairs. 

\begin{defn}
Let $(X,\Delta)$ be a sub-pair. A \emph{non-lc place} of $(X,\Delta)$ is a prime divisor $D$ over $X$ such that $a(D,X,\Delta)<0$. A \emph{non-lc center} of $(X,\Delta)$ is the center of a non-lc place of $(X,\Delta)$ on $X$. The \emph{non-lc locus} $\Nlc(X,\Delta)$ of $(X,\Delta)$ is the union of all non-lc centers of $(X,\Delta)$. 
\end{defn}

Some of the notation and results below are adopted from the theory of \emph{quasi-log varieties}. Although there are many papers in this direction, we will only use the results in \cite{Amb03,Fuj11}, and we will always translate them into the language of (not necessarily lc) pairs. To make these translations valid, we only need to recall the following result:
\begin{lem}[cf. {\cite[Example 4.3.1]{Amb03}}]
Let $(X,\Delta)$ be a pair. Then $(X,\Delta)$ can be considered as a quasi-log variety $[X,K_X+\Delta]$, such that $\Nlc(X,\Delta)$ is exactly the non-qlc locus of $[X,K_X+\Delta]$.
\end{lem}

\begin{defn}[cf. {\cite[Definition 5.2]{Amb03}, \cite[Theorem 4.5.2(1), Definition 6.7.1]{Fuj11}}]
Let $(X,\Delta)$ be a (not necessarily lc) pair. We define
$$\overline{NE}(X/U)_{\Nlc(X,\Delta)}:=\Ima(\overline{NE}(\Nlc(X,\Delta)/U)\rightarrow\overline{NE}(X/U)).$$
\end{defn}

\begin{defn}[{cf. \cite[Definition 5.3]{Amb03}, \cite[Definition 6.7.2]{Fuj11}}]
Let $(X,\Delta)$ be a (not necessarily lc) pair and $\pi: X\rightarrow U$ a projective morphism. Let $F$ be an extremal face of $\overline{NE}(X/U)$.
\begin{enumerate}
\item A \emph{supporting function} of $F$ is a  $\pi$-nef $\Rr$-divisor $H$ such that $F=\overline{NE}(X/U)\cap H^{\bot}$. If $H$ is a $\Qq$-divisor, we say that $H$ is a \emph{rational supporting function}. Since $F$ is an extremal face of $\overline{NE}(X/U)$, $F$ always has a supporting function.
    \item For any $\Rr$-Cartier $\Rr$-divisor $D$ on $X$, we say that $F$ is $D$-\emph{negative} if $$F\cap\overline{NE}(X/U)_{D\geq 0}=\{0\}.$$
    \item We say that $F$ is \emph{rational} if $F$ has a rational supporting function.
    \item We say that $F$ is \emph{relatively ample at infinity with respect to} $(X,\Delta)$  if $$F\cap\overline{NE}(X/U)_{\Nlc(X,\Delta)}=\{0\}.$$ Equivalently, $H|_{\Nlc(X,\Delta)}$ is $\pi|_{\Nlc(X,\Delta)}$-ample for any supporting function $H$ of $F$.
    \item We say that $F$ is \emph{contractible at infinity with respect to} $(X,\Delta)$ if $F$ has a rational supporting function $H$ and $H|_{\Nlc(X,\Delta)}$ is $\pi|_{\Nlc(X,\Delta)}$-semi-ample.
\end{enumerate}
\end{defn}

\begin{thm}[Cone theorem for not necessarily lc pairs, {cf. \cite[Theorem 5.10]{Amb03}, \cite[Theorems 4.5.2, 6.7.4]{Fuj11}}]\label{thm: cone theorem for not necessarily lc pairs}
Let $(X,\Delta)$ be a (not necessarily lc) pair and $\pi: X\rightarrow U$ a projective morphism. Let $\{R_j\}_{j\in\Lambda}$ be the set of $(K_X+\Delta)$-negative extremal rays in $\overline{NE}(X/U)$ that are rational and relatively ample at infinity with respect to $(X,\Delta)$. Then:
\begin{enumerate}
    \item $$\overline{NE}(X/U)=\overline{NE}(X/U)_{K_X+\Delta\geq 0}+\overline{NE}(X/U)_{\Nlc(X,\Delta)}+\sum_{j\in\Lambda} R_j.$$
    \item Each $R_j$ is spanned by a rational curve $C_j$ such that $\pi(C_j)=\{pt\}$ and $$0<-(K_X+\Delta)\cdot C_j\leq 2\dim X.$$
    \item For any $\pi$-ample $\Rr$-divisor $A$ on $X$, 
    $$\Lambda_A:=\{j\in\Lambda\mid R_j\subset\overline{NE}(X/U)_{K_X+\Delta+A<0}\}$$
    is a finite set. In particular, $\{R_j\}_{j\in\Lambda}$ is a discrete subset in $\overline{NE}(X/U)_{K_X+\Delta<0}$, and we may write
    $$\overline{NE}(X/U)=\overline{NE}(X/U)_{K_X+\Delta+A\geq 0}+\overline{NE}(X/U)_{\Nlc(X,\Delta)}+\sum_{j\in\Lambda_A}R_j.$$
    \item Let $F$ be a $(K_X+\Delta)$-negative extremal face in $\overline{NE}(X/U)$ that is relatively ample at infinity with respect to $(X,\Delta)$. Then $F$ is a rational extremal face, and is contractible at infinity with respect to $(X,\Delta)$.
\end{enumerate}
\end{thm}

\begin{thm}[Base-point-free theorem for not necessarily lc pairs, {cf. \cite[Theorem 5.3]{Amb03}, \cite[Theorems 4.5.5, 6.5.1]{Fuj11}}]\label{thm: base-point-free theorem for not necessarily lc pairs}
Let $(X,\Delta)$ be a (not necessarily lc) pair and $\pi: X\rightarrow U$ a projective morphism. Let $L$ be a $\pi$-nef Cartier divisor on $X$. Assume that
\begin{enumerate}
    \item $qL-(K_X+\Delta)$ is $\pi$-ample for some real number $q>0$, and
    \item $mL|_{\Nlc(X,\Delta)}$ is $\pi|_{\Nlc(X,\Delta)}$-generated for any $m\gg 0$,
\end{enumerate}
then $mL$ is $\pi$-generated for any $m\gg 0$. In particular, $L$ is $\pi$-semi-ample.
\end{thm}

We also include the following contraction theorem for the sake of completeness.
\begin{thm}[Contraction theorem for not necessarily lc pairs, {cf. \cite[Theorem 5.6, Lemma 6.3]{Amb03},  \cite[Theorems 4.5.2(4), 6.7.3]{Fuj11}}]\label{thm: contraction theorem for not necessarily lc pairs}
Let $(X,\Delta)$ be a (not necessarily lc) pair and $\pi: X\rightarrow U$ a projective morphism. Let $H$ be a $\pi$-nef Cartier divisor on $X$, $F:=\overline{NE}(X/U)\cap H^{\bot}$ an extremal face of $\overline{NE}(X/U)$, such that $F$ is $(K_X+\Delta)$-negative and contractible at infinity with respect to $(X,\Delta)$. Then there exists a projective morphism $\cont_F: X\rightarrow Y$ over $U$ satisfying the following:
\begin{enumerate}
    \item Let $C$ be an integral curve such that $\pi(C)$ is a point. Then $\cont_F(C)$ is a point if and only if $[C]\in F$.
    \item $\mathcal{O}_Y\cong(\cont_F)_*\mathcal{O}_X$. In other words, $\cont_F$ is a contraction.
    \item Let $L$ be a line bundle on $X$ such that $L\cdot C=0$ for every curve $C$ such that $[C]\in F$. Assume that $L^{\otimes m}|_{\Nlc(X,\Delta)}$ is $\cont_F|_{\Nlc(X,\Delta)}$-generated for every $m\gg 0$. Then there exists a line bundle $L_Y$ on $Y$ such that $L\cong(\cont_F)^*L_Y$.
\end{enumerate}
\end{thm}

\subsection{Sub-adjunction}

We need the following sub-adjunction result for NQC glc g-pairs:

\begin{thm}[Generalized sub-adjunction, {cf. \cite[Theorem 5.1]{HL19}}]\label{thm: gpair subadjunction}
Let $(X,B,\Mm)/U$ be an NQC glc g-pair, $\tilde W$ a glc center of $(X,B,\Mm)$, and $W$ the normalization of $\tilde W$. Let $\iota: W\rightarrow X$ be the induced morphism. Then there exists an NQC glc g-pair $(W,B_W,\Mm^W)/U$ on $W$, such that
$$K_{W}+B_W+\Mm^W_{W}\sim_{\Rr}\iota^*(K_X+B+\Mm_X).$$
\end{thm}
\begin{rem}
Note that the fact that $(W,B_W,\Mm_W)/U$ is NQC is not written down in the statement of \cite[Thereom 5.1]{HL19}, but it is stated on line -3 of Page 25 of \cite{HL19}.
\end{rem}

\subsection{Proof of the cone theorem}

In this subsection, we prove the cone theorem (Theorem \ref{thm: cone theorem glc g-pairs}).  We first prove a useful lemma which allows us to associate a generalized lc pair with a (not necessarily lc) pair.

\begin{lem}\label{lem: perturb gpair to make nlc locus ngklt locus}
Let $(X,B,\Mm)/U$ be a glc g-pair and $A$ a nef and big$/U$ $\Rr$-divisor on $X$. Then there exists a pair $(X,\Delta)$, such that
\begin{enumerate}
    \item $\Delta\sim_{\Rr,U}B+\Mm_X+A$, and
    \item $\Nlc(X,\Delta)=\Ngklt(X,B,\Mm)$.
\end{enumerate}
\end{lem}
\begin{proof}
Let $h: W\rightarrow X$ be a log resolution of $(X,\Supp B)$ such that $\Mm$ descends to $W$, and suppose that
$$K_W+B_W+\Mm_W=h^*(K_X+B+\Mm_X)$$
for some sub-glc g-sub-pair $(W,B_W,\Mm)/U$. Since $\Mm_W$ is nef$/U$, $\Mm_W+h^*A$ is nef and big$/U$. Thus there exists an $\Rr$-divisor $E\geq 0$ such that
$$\Mm_W+h^*A=H_n+\frac{1}{n}E$$
for any positive integer $n$ and some ample$/U$ $\Rr$-divisors $H_n$ on $W$. Since $h: W\rightarrow X$ is a log resolution of $(X,\Supp B)$, we may pick $n\gg 0$ such that $\Nlc(W,B_W+\frac{1}{n}E)\subset\Supp B_W^{=1}$. In particular, for any positive real number $\epsilon$, $\Nlc(W,B_W+\epsilon B_W^{=1}+\frac{1}{n}E)=\Supp B_W^{=1}$.

Now we may pick a real number $0<\epsilon_0\ll 1$ such that $H_n-\epsilon_0 B_W^{=1}$ is ample$/U$. Then we may pick $0\leq A_W\sim_{\Rr,U}H_n-\epsilon_0 B_W^{=1}$ such that $(W,\Delta_W:=B_W+\epsilon_0 B_W^{=1}+\frac{1}{n}E+A_W)$ is a sub-pair and $\Nlc(W,\Delta_W)=\Supp B_W^{=1}$.

$(X,\Delta:=h_*\Delta_W)$ satisfies our requirements.
\end{proof}

\begin{lem}\label{lem: low dimension gpair cone theorem imply finite extremal rays}
Let $d\geq 2$ be an integer. Assume Theorem \ref{thm: cone theorem glc g-pairs} in dimension $\leq d-1$.

Let $(X,B,\Mm)/U$ be an NQC glc g-pair of dimension $d$ and $\pi: X\rightarrow U$ the associated projective morphism. Let $A$ be an ample$/U$ $\Rr$-divisor on $X$ and $\{R_j\}_{j\in\Lambda'_A}$ the set of $(K_X+B+\Mm_X+A)$-negative extremal rays (that are not necessarily rational) in $\overline{NE}(X/U)$. Then:
\begin{enumerate}
    \item $\Lambda'_A$ is a finite set. In particular,
    $$\overline{NE}(X/U)=\overline{NE}(X/U)_{K_X+B+\Mm_X+A\geq 0}+\sum_{j\in\Lambda'_A}R_j.$$
    \item For any $j\in\Lambda_A'$, $R_j$ is spanned by a rational curve $C_j$ such that $\pi(C_j)=\{pt\}$ and
    $$0<-(K_X+B+\Mm_X+A)\cdot C_j\leq 2\dim X.$$
\end{enumerate}
\end{lem}
\begin{proof}
By Lemma \ref{lem: perturb gpair to make nlc locus ngklt locus}, we may pick $0\leq\Delta\sim_{\Rr,U}B+\Mm_X+A$ such that $\Nlc(X,\Delta)=\Ngklt(X,B,\Mm)$.

For any glc center $\tilde W$ of $(X,B,\Mm)$ with normalization $W$, we let $(W,B_{W},\Mm^{W})/U$ be the NQC glc g-pair given by the sub-adjunction
    $$K_{W}+B_{W}+\Mm^{W}_{W}\sim_{\Rr}(K_X+B+\Mm_X)|_{W}$$
    as in Theorem \ref{thm: gpair subadjunction}, and let $A_W:=A|_W$.  By Theorem \ref{thm: cone theorem glc g-pairs} in dimension $\leq d-1$, we have 
$$\overline{NE}(W/U)=\overline{NE}(W/U)_{K_W+B_W+\Mm^W_W+A_W\geq 0}+\sum_{j\in\Lambda_{A_W}}R_{j,W},$$
where $\{R_{j,W}\}_{j\in\Lambda_{A_W}}$ is the set of $(K_W+B_W+\Mm^W_W+A_W)$-negative extremal rays in $\overline{NE}(W/U)$ that are rational, where $\Lambda_{A_W}$ is a finite set. For any $j\in\Lambda_{A_W}$, we let $R_j$ be the image of $R_{j,W}$ in $X$ under the map
$$\cup_W\overline{NE}(W/U)\rightarrow\overline{NE}(\Nlc(X,\Delta)/U)\rightarrow\overline{NE}(X/U)$$
and let $\Lambda^0_A:=\cup_{W}\Lambda_{A_W}$. Then $\Lambda^0_A$ is a finite set. Finally, we let $\{R_j\}_{j\in\Lambda^1_A}$ be the set of $(K_X+B+\Mm_X+A)$-negative extremal rays in $\overline{NE}(X/U)$ that are relatively ample at infinity with respect to $(X,\Delta)$. By Theorem \ref{thm: cone theorem for not necessarily lc pairs}(3), $\Lambda_A^1$ is a finite set.
\begin{claim}\label{claim: gpair necone spanned by images first step}
$$\overline{NE}(X/U)=\overline{NE}(X/U)_{K_X+B+\Mm_X+A\geq 0}+\sum_{j\in\Lambda^0_A}R_j+\sum_{j\in\Lambda^1_A}R_j.$$
\end{claim}
\begin{proof}
For simplicity, we let $$V:=\overline{NE}(X/U)_{K_X+B+\Mm_X+A\geq 0}+\sum_{j\in\Lambda^0_A}R_j+\sum_{j\in\Lambda^1_A}R_j.$$
For any curve $C$ on $X$, we will write $[C]$ for its class in $\overline{NE}(X/U)$, and for any glc center $\tilde W$ of $(X,B,\Mm)$ with normalization $W$, if $C\subset W$, then we will write $[C]_W$ for its class in $\overline{NE}(W/U)$.

Suppose that $\overline{NE}(X/U)\not=V.$ By Theorem \ref{thm: cone theorem for not necessarily lc pairs}(1), we have
$$\overline{NE}(X/U)=\overline{NE}(X/U)_{K_X+B+\Mm_X+A\geq 0}+\overline{NE}(X/U)_{\Nlc(X,\Delta)}+\sum_{j\in\Lambda^1_A}R_j.$$
Thus there exists an integral curve $C\subset\Nlc(X,\Delta)=\Ngklt(X,B,\Mm)$, such that $[C]$ is not contained in $V$. We may write 
$$C=\sum_{W\mid W\text{ is a glc center of }(X,B,\Mm)}C_W,$$ 
where each $C_W$ is an integral curve in $W$. For any $C_W$, we have
$$[C_W]_W=c^0_{W}R^0_W+\sum_{j\in\Lambda_{A_W}}c_{j,W}R_{j,W}$$
where $c^0_{W}$ and each $c_{j,W}$ are non-negative real numbers, and $R_W^0\in \overline{NE}(W/U)_{K_W+B_W+\Mm^W_W+A_W\geq 0}$. Since the image of $R^0_W$ in $X$ is contained in $\overline{NE}(X/U)_{K_X+B+\Mm_X+A\geq 0}$, $[C_W]$ is contained in $\overline{NE}(X/U)_{K_X+B+\Mm_X+A\geq 0}+\sum_{j\in\Lambda^0_A}R_j$. Thus $[C_W]$ is contained in $V$, hence $[C]$ is contained in $V$, a contradiction.  
\end{proof}
\noindent\textit{Proof of Lemma \ref{lem: low dimension gpair cone theorem imply finite extremal rays} continued}. By Claim \ref{claim: gpair necone spanned by images first step}, any $(K_X+B+\Mm_X+A)$-negative extremal ray in $\overline{NE}(X/U)$ must be contained in $\{R_j\}_{j\in\Lambda^0_A\cup\Lambda^1_A}$, so $\Lambda'_A\subset\Lambda^0_A\cup\Lambda^1_A$. Since $\Lambda^0_A\cup\Lambda^1_A$ is a finite set, $\Lambda'_A$ is a finite set, and we get (1).

By Theorem \ref{thm: cone theorem glc g-pairs} in dimension $\leq d-1$, for any $j\in\Lambda_{A_W}$, $R_{j,W}$ is spanned by a rational curve $C_j$ such that the image of $C_j$ in $U$ is a point, and
$$0<-(K_W+B_W+\Mm^W_W+A_W)\cdot C_j\leq 2\dim W<2\dim X.$$
Therefore, for any $j\in\Lambda^0_A=\cup_W\Lambda_{A_W}$, $R_j$ is spanned by the curve $C_j$ such that $\pi(C_j)=\{pt\}$ and
$$0<-(K_X+B+\Mm_X+A)\cdot C_j\leq 2\dim X.$$
By Theorem \ref{thm: cone theorem for not necessarily lc pairs}(2), for any $j\in\Lambda^1_A$, $R_j$ is spanned by a rational curve $C_j$ such that $\pi(C_j)=\{pt\}$ and
$$0<-(K_X+B+\Mm_X+A)\cdot C_j\leq 2\dim X.$$
 Thus (2) holds and the proof is complete.
\end{proof}

\begin{lem}\label{lem: gpair cone theorem spanned by extremal rays rational case}
Let $d\geq 2$ be an integer. Assume Theorem \ref{thm: cone theorem glc g-pairs} in dimension $\leq d-1$.

Let $(X,B,\Mm)/U$ be an NQC glc $\Qq$-g-pair of dimension $d$ and $\pi: X\rightarrow U$ the associated projective morphism. Let $A$ be an ample$/U$ $\Qq$-divisor on $X$ and $\{R_j\}_{j\in\Lambda_A}$ the set of $(K_X+B+\Mm_X+A)$-negative extremal rays in $\overline{NE}(X/U)$ that are rational. Then $\Lambda_A$ is a finite set, and
$$\overline{NE}(X/U)=\overline{NE}(X/U)_{K_X+B+\Mm_X+A\geq 0}+\sum_{j\in\Lambda_A}R_j.$$
\end{lem}
\begin{proof}
We may assume that $\dim_{\Rr}N^1(X/U)\geq 2$, otherwise there is nothing to prove.

By Lemma \ref{lem: low dimension gpair cone theorem imply finite extremal rays}, the number of $(K_X+B+\Mm_X+A)$-negative extremal rays in $\overline{NE}(X/U)$ is a finite set, so $\Lambda_A\subset\Lambda_A'$ is a finite set.

For simplicity, we let $V:=\overline{NE}(X/U)_{K_X+B+\Mm_X+A\geq 0}+\sum_{j\in\Lambda_A}R_j$. Suppose that $V\not=\overline{NE}(X/U)$. Since $\dim_{\Rr}N^1(X/U)\geq 2$, there exists a Cartier divisor $N$ on $X$ satisfying the following:
\begin{itemize}
    \item $N$ is not numerically equivalent to a multiple of $K_X+B+\Mm_X+A$ over $U$,
    \item $N$ is positive on $V\backslash\{0\}$, and
    \item $N\cdot z_0<0$ for some $z_0\in\overline{NE}(X/U)$.
\end{itemize}
Let $Q$ be the dual cone of $\overline{NE}(X/U)_{K_X+B+\Mm_X+A\geq 0}$, i.e.,
$$Q=\{D\in N^1(X/U)\mid D\cdot z\geq 0\text{ for any }z\in\overline{NE}(X/U)_{K_X+B+\Mm_X+A\geq 0}\},$$
then $Q$ is generated by $\pi$-nef divisors and $K_X+B+\Mm_X+A$. Since $N$ is positive on $\overline{NE}(X/U)_{K_X+B+\Mm_X+A\geq 0}\backslash\{0\}$, $N$ is in the interior of $Q$. By Kleiman's Criterion, there exists an ample$/U$ $\Qq$-divisor $H$ on $X$ and a positive real number $p$, such that 
$$N=H+p(K_X+B+\Mm_X+A).$$
Since $N\cdot z_0<0$ and $H$ is ample$/U$, we may let
$$t:=\sup\{s\mid H+s(K_X+B+\Mm_X+A)\text{ is nef}/U\}.$$
Then $0<t<p$. Since $(H+t(K_X+B+\Mm_X+A))\cdot z\geq 0$ for any $z\in\overline{NE}(X/U)_{K_X+B+\Mm_X+A\geq 0}$, by Lemma \ref{lem: low dimension gpair cone theorem imply finite extremal rays},
$$t=\max\{s\mid (H+s(K_X+B+\Mm_X+A))\cdot R_j\geq 0,\forall j\in\Lambda_A'\}$$
where $\{R_j\}_{j\in\Lambda'_A}$ the set of $(K_X+B+\Mm_X+A)$-negative extremal rays in $\overline{NE}(X/U)$ and is a finite set. Thus $t$ is a rational number. Since $N$ is not a multiple of $K_X+B+\Mm_X+A$, $H+t(K_X+B+\Mm_X+A)$ is a rational supporting function of a $(K_X+B+\Mm_X+A)$-negative extremal face $F_N$, which is spanned by $(K_X+B+\Mm_X+A)$-negative extremal rays. By Lemma \ref{lem: low dimension gpair cone theorem imply finite extremal rays}, $F_N$ is spanned by finitely many $(K_X+B+\Mm_X+A)$-negative extremal rays $R^1,\dots,R^n$ in $\overline{NE}(X/U)$ for some positive integer $n$. In particular, we may pick a Cartier divisor $L$ on $X$ such that $L\cdot R^1>0$ and $L\cdot R^i<0$ for any $i\geq 2$. Since $H$ is ample$/U$ and $N$ is not numerically equivalent to a multiple of $K_X+B+\Mm_X+A$ over $U$, we may pick a rational number $\epsilon\in (0,1)$ such that
\begin{itemize}
    \item $N_{\epsilon}:=(H-\epsilon L)+p(K_X+B+\Mm_X+A)$ is not numerically equivalent to a multiple of $K_X+B+\Mm_X+A$ over $U$ for any $\epsilon\in (0,\epsilon_0)$,
    \item $H-\epsilon_0 L$ is ample$/U$, and
    \item $N_{\epsilon_0}\cdot z_0<0$.
\end{itemize}
Thus $N_{\epsilon}$ is positive on $\overline{NE}(X/U)_{K_X+B+\Mm_X+A\geq 0}$. Since $\Lambda_A$ is a finite set and $N\cdot R_j>0$ for any $j\in\Lambda_A$, we may pick a rational number $\epsilon_1\in (0,\epsilon_0)$ such that $N_{\epsilon_1}\cdot R_j>0$ for any $j\in\Lambda_A$. In particular, $N_{\epsilon_1}$ is positive on $V\backslash\{0\}$. Now we let
$$t_1:=\sup\{s\mid H-\epsilon_1L+s(K_X+B+\Mm_X+A)\text{ is nef}/U\}.$$
By our construction,
$$t_1=\frac{(H-\epsilon_1L)\cdot R^1}{-(K_X+B+\Mm_X+A)\cdot R^1}$$
is a rational number, $0<t_1<t<p$, and $H-\epsilon_1L+t_1(K_X+B+\Mm_X+A)$ is a rational supporting function of $R^1$. Thus $R^1\in\Lambda_A$, and so $N_{\epsilon_1}\cdot R^1>0$. Therefore, $p<t_1$, a contradiction.
\end{proof}

\begin{lem}\label{lem: gpair cone theorem spanned by extremal rays real case}
Let $d\geq 2$ be an integer. Assume Theorem \ref{thm: cone theorem glc g-pairs} in dimension $\leq d-1$.

Let $(X,B,\Mm)/U$ be an NQC glc g-pair of dimension $d$ and $\pi: X\rightarrow U$ the associated projective morphism. Let $A$ be an ample$/U$ $\Rr$-divisor on $X$ and $\{R_j\}_{j\in\Lambda_A}$ the set of $(K_X+B+\Mm_X+A)$-negative extremal rays in $\overline{NE}(X/U)$ that are rational. Then $\Lambda_A$ is a finite set, and
$$\overline{NE}(X/U)=\overline{NE}(X/U)_{K_X+B+\Mm_X+A\geq 0}+\sum_{j\in\Lambda_A}R_j.$$
In particular, any $(K_X+B+\Mm_X+A)$-negative extremal ray in $\overline{NE}(X/U)$ is rational.
\end{lem}
\begin{proof}
Let $\{R_j\}_{j\in\Lambda_{A'}}$ be the set of $(K_X+B+\Mm_X+A)$-negative extremal rays (that are not necessarily rational) in $\overline{NE}(X/U)$. By Lemma \ref{lem: low dimension gpair cone theorem imply finite extremal rays}, $\Lambda_A'$ is a finite set, and we have
    $$\overline{NE}(X/U)=\overline{NE}(X/U)_{K_X+B+\Mm_X+A\geq 0}+\sum_{j\in\Lambda'_A}R_j.$$
By Theorem \ref{thm: shokurov polytope gpair}, there exist real numbers $a_1,\dots,a_k\in (0,1]$, such that
\begin{itemize}
\item $\sum_{i=1}^ka_i=1$,
    \item $K_X+B=\sum_{i=1}^k a_i(K_X+B^i)$ and $\Mm=\sum_{i=1}^ka_i\Mm^i$, and
    \item $(X,B^i,\Mm^i)/U$ is a glc $\Qq$-g-pair for each $i$.
\end{itemize}
Let $A=\sum_{i=1}^cr_iA_i$, where $r_1,\dots,r_c>0$ are real numbers such that $r_1,\dots,r_c$ are linearly independent over $\Qq$, and $A_1,\dots,A_c$ are ample$/U$ $\Qq$-divisors.

Since $\Lambda_A'$ is a finite set, we may pick rational numbers $\bar a_1,\dots,\bar a_k\in (0,1]$ and $\bar r_1,\dots,\bar r_c>0$, such that $\sum_{i=1}^k\bar a_i=1$, each $\bar a_i$ is sufficiently close to $a_i$ and each $\bar r_i$ is sufficiently close to $r_i$, such that 
\begin{itemize}
    \item $(X,\bar B:=\sum_{i=1}^k\bar a_iB_i,\bar\Mm:=\sum_{i=1}^k\bar a_i\Mm^i)$ is glc,
    \item $\bar A:=\sum_{i=1}^c\bar r_iA_i$ is ample, and
    \item $(K_X+\bar B+\bar \Mm_X+\bar A)\cdot R_j<0$ for any $j\in\Lambda_A'$.
\end{itemize}
 By Lemma \ref{lem: gpair cone theorem spanned by extremal rays rational case}, we have
$$\overline{NE}(X/U)=\overline{NE}(X/U)_{K_X+\bar B+\bar \Mm_X+\bar A}+\sum_{j\in\Phi} R_j,$$
where $\{R_j\}_{j\in\Phi}$ is the set of $(K_X+\bar B+\bar \Mm_X+\bar A)$-negative extremal rays in $\overline{NE}(X/U)$. Moreover, $R_j$ is rational for any $j\in\Phi$. By our construction, $\Lambda_A'\subset\Phi$. Thus $R_j$ is rational for any $j\in\Lambda_A'$, hence $\Lambda_A=\Lambda_A'$ and we are done.
\end{proof}

\begin{proof}[Proof of Theorem \ref{thm: cone theorem glc g-pairs}]
We apply induction on dimension of $X$. The $\dim X=1$ case is obviously true. So we may assume that $\dim X=d$ where $d\geq 2$ is an integer and Theorem \ref{thm: cone theorem glc g-pairs} holds in dimension $\leq d-1$.

For any $(K_X+B+\Mm_X)$-negative extremal ray $R$ in $\overline{NE}(X/U)$, $R$ is also a $(K_X+B+\Mm_X+A)$-negative extremal ray for some ample$/U$ $\Rr$-divisor $A$ on $X$. By Lemma \ref{lem: gpair cone theorem spanned by extremal rays real case}, $R$ is rational. By Lemma \ref{lem: low dimension gpair cone theorem imply finite extremal rays}(2), $R$ is generated by a rational curve $C$ such that $\pi(C)=\{pt\}$ and $$0<-(K_X+B+\Mm_X+A)\cdot C\leq 2\dim X.$$ Since $R$ is also a $(K_X+B+\Mm_X+\epsilon A)$-negative extremal ray for any $\epsilon\in (0,1)$, by Lemma \ref{lem: low dimension gpair cone theorem imply finite extremal rays}(2) again, we may assume that $$0<-(K_X+B+\Mm_X+\epsilon A)\cdot C\leq 2\dim X$$ for any $\epsilon\in (0,1)$. Thus 
$$0<-(K_X+B+\Mm_X)\cdot C\leq 2\dim X,$$ 
and we get (2). (3) follows from Lemma \ref{lem: gpair cone theorem spanned by extremal rays real case} and the fact that $$\{R_j\}_{j\in\Lambda}\subset\cup_{n=1}^{+\infty}\{R_j\}_{j\in\Lambda_{\frac{1}{n}A}}$$ for any ample$/U$ $\Rr$-divisor $A$ on $X$. (1) follows from (3). 

We now prove (4). For any $(K_X+B+\Mm_X)$-negative extremal face $F$ in $\overline{NE}(X/U)$, $F$ is also a $(K_X+B+\Mm_X+A)$-negative extremal face for some ample$/U$ $\Rr$-divisor $A$ on $X$. Let $V:=F^\bot\subset N^1(X/U)$. Then since $F$ is spanned by a subset of $\{R_j\}_{j\in\Lambda_A}$, $V$ is defined over $\Qq$. We let
$$W_F:=\overline{NE}(X/U)_{K_X+B+\Mm_X+A\geq 0}+\sum_{j\mid j\in\Lambda_A,R_j\not\subset F}R_j.$$
Then $W_F$ is a closed cone, $\overline{NE}(X/U)=W_F+F$, and $W_F\cap F=\{0\}$. The supporting functions of $F$ are the elements in $V$ that are positive on $W_F\backslash\{0\}$, which is a non-empty open subset of $V$, and hence contains a rational element $H$. In particular, $F=H^\bot\cap \overline{NE}(X/U)$, hence $F$ is rational, and we get (4).
\end{proof}

\subsection{Proof of the base-point-free theorem and the contraction theorem}

Now we prove the base-point-free theorem (Theorem \ref{thm: base-point-free theorem for glc pairs}) for glc g-pairs. First we prove an auxiliary lemma.

 \begin{lem}\label{lem: special extraction}
 Let $(X,B,\Mm)/U$ be a glc g-pair such that $\Mm_X$ is $\Rr$-Cartier and $\Ngklt(X,B,\Mm)=\Nklt(X,B)$. Then there exists a birational morphism $h: W\rightarrow X$ such that $\Mm$ descends to $W$ and $\Supp(h^*\Mm_X-\Mm_W)=\Exc(h)$.
 \end{lem}
 
 \begin{proof}
Let $f: Y\rightarrow X$ be a log resolution of $(X,B)$ such that $\Mm$ descends to $Y$. Let $F=\Exc(f)$ be the reduced exceptional divisor. Write $K_Y+f^{-1}_*B+G=f^*(K_X+B)$ and $\mathbf M_Y+E=f^*\mathbf M_X$. Write $\Supp E=\cup_i E_i$. Note that $E=f^{-1}(f(E))$. If this is not the case, then since the fibers of $f$ are connected, there is a curve $C$ contained in a fiber $f^{-1}(x)$ such that $C$ intersects the support of $E$ but is not contained in the support of $E$. But then $-E\cdot C<0$ contradicting the fact that $-E$ is nef over $X$. Let $Y^0=Y\setminus\Supp E$ and let $X^0=X\setminus f(\Supp E)$, then $Y^0=f^{-1}(X^0)$. 

Since $\Ngklt(X,B,\Mm)=\Nklt(X,B)$, the support of $E$ does not contain any strata of $G^{=1}$. In particular $E \wedge G^{=1}=0$, and any element in $\Ngklt(X,B,\mathbf M)$ is contained in $X\setminus X^0$. 
       
       We now consider the generalized pair

      \[ (Y,f^{-1}_*B+e G^{=1}+(1-e)F+\sum_is_iE_i, {t\overline{\mathbf M}_Y})/X\]
where $0<s_i\ll e\ll 1$, $t\gg 1$, and the real numbers $s_i$ are sufficiently general (i.e. their representatives in $\mathbb R/\mathbb Q$ are sufficiently general). 
We have
$$K_Y+f^{-1}_*B+eG^{=1}+(1-e)F+\sum_is_iE_i+t\mathbf M_Y\sim_{\mathbb R,X}eG^{=1} +(1-e)F-G-tE+\sum_is_iE_i\sim _{\mathbb R,X} F'-E'$$
where the coefficients of $E'$ are sufficiently general real numbers, 
${\rm Supp} E'={\rm Supp} E$, and ${\rm Supp} F'$ consists of the set of exceptional divisors not contained in the support of $E\vee G^{=1}$.

We will now apply Theorem \ref{thm: existence of glc closure} to this generalized pair. To check the hypothesis, we consider the open subset $Y^0$ and $X^0$ defined above. (1) clearly holds, (3) has been checked above, and (4) holds since $\mathbf M _{Y}|_{Y^0}=(f|_{Y^0})^*\mathbf M_X|_{X^0}$ as $E|_{Y^0}=0$. For (2), we must check that \[(Y^0, (f^{-1}_*B+e G^{=1}+(1-e)F+\sum_is_iE_i)|_{Y^0}, {t\overline{\mathbf M}_Y}|_{Y^0})=(Y^0, (f^{-1}_*B+e G^{=1}+(1-e)F)|_{Y^0}, 0)\] has a good minimal model over $X^0$. Since $K_{Y^0}+(f^{-1}_*B+e G^{=1}+(1-e)F)|_{Y^0}\sim _{\mathbb R, X^0} F'|_{Y^0}$ where $ F'|_{Y^0}$ is effective and exceptional over $X^0$, by Lemma \ref{lem: rlinear version of hl18 3.8}, $(Y^0, (f^{-1}_*B+e G^{=1}+(1-e)F)|_{Y^0})/X^0$ has a good minimal model and (2) holds.
Therefore, by Theorems \ref{thm: existence of glc closure} and \ref{thm: mmp with scaling gpair terminates assuming gmm},
we can run a $(K_Y+f^{-1}_*B+e G^{=1}+(1-e)F+\sum_is_iE_i+{t{\mathbf M}_Y})$-MMP/$X$: $Y\dashrightarrow Z$ which contracts $F'$ and get a good minimal model$/X$. 

By \cite[Lemma 4.4(3)]{BZ16}, $\overline{\Mm}_Y$ descends to $Z$, hence $\Mm$ descends to $Z$. Let $E_Z$, $E'_Z$ be the strict transforms of $E$, $E'$ on the minimal model $Z$ respectively. Then $-E'_Z$ is 
semi-ample/$X$ and we can then take the corresponding ample model $g: Z\to W$ of $-E'_Z/X$.
Since $-E'_W$ is ample over $X$, the only $h:W\to X$ exceptional divisors are the components of $-E'_W$.

Since the coefficients of 
$E'_Z$ are sufficiently general, no component of $\Supp E'_Z=\Supp E_Z$ is contracted by $h:Z\to W$. To see this, note that if $E'_Z\cdot C = 0$ for any curve $C$ over $X$, then the same is true for every component of $E'_Z$ (since the coefficients of 
$E'_Z$ are sufficiently general). 
Since $E'_Z\equiv_W 0$, it follows that $P\equiv _W 0$ for any component $P$ of the support of $E'_Z$. By the negativity lemma, $P$ is not exceptional.
Note that $g: Z\to W$ is also the ample model of any small perturbation of $-E'_Z$ and so $g_*P$ is $\mathbb Q$-Cartier and $P=g^*g_*P$. 
But then $\Mm_Z\sim_{\Rr,X}-E_Z=-g^*(E_W)$ where $E_W=g_*E_Z$. Thus $\Mm_Z=g^*g_*\Mm_Z=g^*\Mm_W$, so $\Mm$ descends to $W$.

Therefore, $W$ satisfies our requirements.
\end{proof}

\begin{proof}[Proof of Theorem \ref{thm: base-point-free theorem for glc pairs}]

Let $R$ be the $(K_X+B+\Mm_X)$-negative extremal ray such that $L$ is the supporting function of $R$. Then $R$ is also a $(K_X+B+(1-\epsilon)\Mm_X)$-negative extremal ray for some $0<\epsilon\ll 1$. Possibly replacing $\Mm$ with $(1-\epsilon)\Mm$, we may assume that $\Ngklt(X,B,\Mm)=\Nklt(X,B)$. Let $A$ be an ample$/U$ $\Rr$-divisor on $X$ such that $R$ is also $(K_X+B+\Mm_X+A)$-negative extremal ray.

If $\Mm_X\cdot R\geq 0$, then $(K_X+B)\cdot R<0$, and the theorem immediately follows from Theorem \ref{thm: base-point-free theorem for not necessarily lc pairs}. Therefore, we may assume that $\Mm_X\cdot R<0$.

Let $f: Y\rightarrow X$ be a birational morphism such that $\Mm$ descends to $Y$. By the negativity lemma, we may assume that $\Mm_Y=f^*\Mm_X-E$ for some $E\geq 0$ that is exceptional over $X$. 
By Lemma \ref{lem: special extraction}, we may then assume that $\Exc(f)=\Supp E$.

Let $K_Y+B_Y:=f^*(K_X+B)$. By our construction, $\Exc(f)= \Supp E$ does not contain any lc place of $(X,B)$. Thus we may pick $E'\geq 0$ on $Y$ such that $-E'$ is ample$/X$ and $E'$ does not contain any lc place of $(X,B)$. Since $\Ngklt(X,B,\Mm)=\Nklt(X,B)$, we may find $0<\epsilon\ll 1$ such that $f^*A-\epsilon E'$ is ample$/U$ and $(Y,B_Y+\epsilon E')$ is sub-lc. In particular, we may find an ample$/U$ $\Rr$-divisor $0\leq H_Y\sim_{\Rr,U}\Mm_Y+f^*A-\epsilon E'$ on $Y$ such that $(Y,B_Y+H_Y+\epsilon E')$ is sub-lc. Let $\Delta:=B+f_*H_Y$, then $(X,\Delta)$ is lc and $\Delta\sim_{\Rr,U}B+\Mm_X+A$. In particular, $R$ is a $(K_X+\Delta)$-negative extremal ray, and the theorem follows from Theorem \ref{thm: base-point-free theorem for not necessarily lc pairs}.
\end{proof}

The contraction theorem (Theorem \ref{thm: contraction theorem glc g-pairs}) immediately follows from the base-point-free theorem:

\begin{proof}[Proof of Theorem \ref{thm: contraction theorem glc g-pairs}]
By Theorem \ref{thm: cone theorem glc g-pairs}, $R$ has a supporting function $H$ that is a $\pi$-nef Cartier divisor. By Theorem \ref{thm: base-point-free theorem for glc pairs}, $H$ is semi-ample$/U$, hence defines a contraction $\cont_R: X\rightarrow Y$ over $U$. (1) and (2) immediately follow.

Since $-(K_X+B+\Mm_X)$ is ample$/Y$, for any line bundle $L$ on $X$ such that $L\cdot R=0$, $L-(K_X+B+\Mm_X)$ is ample$/Y$. By Theorem \ref{thm: base-point-free theorem for glc pairs}, $mL$ is $\cont_R$-generated and $mL\equiv_Y0$ for any $m\gg 0$. Therefore, $\cont_R$ is defined by $|mL|$ and $|(m+1)L|$ over $Y$ for any $m\gg 0$, which implies that $mL\cong f^*L_{Y,m}$ and $(m+1)L\cong f^*L_{Y,m+1}$ for some line bundles $L_{Y,m}$ and $L_{Y,m+1}$ on $Y$. We may let $L_Y:=L_{Y,m+1}-L_{Y,m}$, and we obtain (3).
\end{proof}

\begin{proof}[Proof of Theorem \ref{thm: cone and contraction theorem glc pair}] It immediately follows from Theorems \ref{thm: contraction theorem glc g-pairs} and \ref{thm: cone theorem glc g-pairs}.
\end{proof}

\subsection{Corollaries} With the cone and contraction theorems proven, we can prove the following three corollaries, which guarantee that negative extremal contractions associated with NQC glc g-pairs behave similarly to negative extremal contractions associated with usual pairs. The statements and proofs are similar to \cite[Corollaries 3.17, 3.18]{KM98}. These corollaries are necessary for us to run the minimal model program.

\begin{cor}\label{cor: gpair negative extremal contraction picard number compare}
Let $(X,B,\Mm)/U$ be a $\Qq$-factorial NQC glc g-pair and $f: X\rightarrow Z$ a contraction of a $(K_X+B+\Mm_X)$-negative extremal ray $R$ over $U$. Then $\rho(X)=\rho(Z)+1$.
\end{cor}
\begin{proof}
$R$ is generated by a curve $C$ by Theorem \ref{thm: cone theorem glc g-pairs}(2). We consider the maps
$$0\rightarrow\Pic(Z)\xrightarrow{D\rightarrow f^*D}\Pic(X)\xrightarrow{L\rightarrow (L\cdot C)}\mathbb Z.$$
We show that the sequence above is an exact sequence.

By Theorem \ref{thm: contraction theorem glc g-pairs}(2), $f$ is a contraction, so $f_*f^*D=D$ for any $D\in\Pic(Z)$, hence $\Pic(Z)\xrightarrow{D\rightarrow f^*D}\Pic(X)$ is an injection. By Theorem \ref{thm: contraction theorem glc g-pairs}(3), for any $L\in\Pic(X)$, if $L\cdot C=0$, then $L\cong f^*L_Y$ for some  line bundle $L_Y$ in $Y$. In particular, $L$ and $f^*L_Y$ correspond to the same element in $\Pic(X)$. Thus the sequence above is exact, and we have $\rho(X)=\rho(Z)+1$.
\end{proof}

\begin{cor}\label{cor: gpair divisorial contraction q factoriality}
Let $(X,B,\Mm)/U$ be a $\Qq$-factorial NQC glc g-pair and $f: X\rightarrow Z$ a contraction of a $(K_X+B+\Mm_X)$-negative extremal ray $R$ over $U$. Assume that $f$ is a divisorial contraction, i.e. $\dim X=\dim Z$ and the exceptional locus of $f$ is an irreducible divisor. Then $Z$ is $\Qq$-factorial.
\end{cor}
\begin{proof}
Let $D_Z$ be an $\Rr$-divisor on $Z$ and $E$ the exceptional divisor of $f$. Let $D$ be the strict transform of $D_Z$ on $X$. Then there exists a real number $t$ such that $(E+tD)\cdot R=0$. By Theorem \ref{thm: contraction theorem glc g-pairs}(3), $E+tD\sim_{\Rr}f^*H$ for some $\Rr$-Cartier $\Rr$-divisor $H$ on $Z$. Thus $D_Z\sim_{\Rr}\frac{1}{t}H$ is $\Rr$-Cartier. Therefore, $Z$ is $\Qq$-factorial.
\end{proof}

\begin{cor}\label{cor: gpair fano contraction q factoriality}
Let $(X,B,\Mm)/U$ be a $\Qq$-factorial NQC glc g-pair and $f: X\rightarrow Z$ a contraction of a $(K_X+B+\Mm_X)$-negative extremal ray $R$ over $U$. Assume that $f$ is a Fano contraction, i.e. $\dim X>\dim Z$. Then $Z$ is $\Qq$-factorial.
\end{cor}
\begin{proof}
Let $D_Z$ be a divisor on $Z$ and $Z^0$ the smooth locus of $Z$. Let $D$ be the closure of $f^{-1}(D_Z|_{Z_0})$. Then $D$ does not intersect any general fiber of $f$, hence $D\cdot R=0$.  By Theorem \ref{thm: contraction theorem glc g-pairs}(3), $D\sim_{\Qq}f^*H$ for some $\Qq$-Cartier $\Qq$-divisor $H$ on $Z$. Thus $D_Z\sim_{\Qq}H$ is $\Qq$-Cartier. Therefore, $Z$ is $\Qq$-factorial.
\end{proof}


The following corollary will allow us to run $\Qq$-factorial generalized MMP with scaling (once the existence of flips is proven in the next section). It is similar to \cite[Lemma 3.19]{HL18}. 

\begin{cor}\label{cor: can run glc mmp with scaling}
Let $(X,B,\Mm)/U$ be a $\Qq$-factorial NQC glc g-pair, $D\geq 0$ an $\Rr$-divisor on $X$, and $\NN$ an NQC$/U$ $\bb$-divisor over $X$, such that $(X,B+D,\Mm+\NN)$ is glc and $K_X+B+D+\Mm_X+\NN_X$ is nef$/U$. Then either $K_X+B+\Mm_X$ is nef$/U$, or there exists an extremal ray $R$ of $\overline{NE}(X/U)$, such that $(K_X+B+\Mm_X)\cdot R<0$ and $(K_X+B+tD+\Mm_X+t\NN_X)\cdot R=0$, where
$$t:=\sup\{s\geq 0\mid K_X+B+sD+\Mm_X+s\NN_X\text{ is nef}/U\}.$$
In particular, $K_X+B+tD+\Mm_X+t\NN_X$ is nef$/U$.
\end{cor}
\begin{proof}
Let $\Delta:=B+D$ and $\PP:=\Mm+\NN$. By Theorem \ref{thm: shokurov polytope gpair}, we may write $K_X+B+\Mm_X=\sum_{i=1}^ka_i(K_X+B^i+\Mm_X^i)$ and $K_X+\Delta+\PP_X=\sum_{i=1}^lc_i(K_X+\Delta^i+\PP_X^i)$, such that 
\begin{itemize}
    \item each $a_i,c_i\in (0,1]$ and $\sum_{i=1}^ka_i=1$, $\sum_{i=1}^lc_i=1$,
    \item $\Mm=\sum_{i=1}^ka_i\Mm^i$ and $\PP=\sum_{i=1}^lc_i\PP^i$,
    \item  each $(X,B^i,\Mm^i)$ is glc, each $(X,\Delta^i,\PP^i)$ is glc, and
    \item each $K_X+\Delta^i+\PP^i_X$ is nef$/U$.
\end{itemize}
Let $m$ be a positive integer such that $m(K_X+B^i+\Mm^i_X)$ and $m(K_X+\Delta^i+\PP_X^i)$ are Cartier for any $i$. 

If $K_X+B+\Mm_X$ is nef then there is nothing left to prove. Therefore, we may assume that $K_X+B+\Mm_X$ is not nef. By Theorem \ref{thm: cone and contraction theorem glc pair}, we may let $\{R_j\}_{j\in\Lambda}$ be the set of $(K_X+B+\Mm_X)$-negative extremal rays in $\overline{NE}(X/U)$, and $C_j$ a curve which generates $R_j$ such that $$-2\dim X\leq (K_X+B+\Mm_X)\cdot C_j<0$$ 
for each $j$. Then for each $j$, we have
$$-2\dim X\leq (K_X+B+\Mm_X)\cdot C_j=\sum_{i=1}^k\frac{a_in_{i,j}}{m}<0$$
and
$$(K_X+\Delta+\PP_X)\cdot C_j=\sum_{i=1}^l\frac{c_in'_{i,j}}{m}\geq 0,$$
where $n_{i,j},n_{i,j}'$ are integers, each $n_{i,j}\geq -2m\dim X$, and each $n_{i,j}'\geq 0$. Therefore, $\Ii:=\{(K_X+B+\Mm_X)\cdot C_j\mid j\in\Lambda\}$ and  $\Ii':=\{(K_X+\Delta+\PP_X)\cdot C_j\mid j\in\Lambda\}$ are DCC sets. 

For any $j\in\Lambda$, let $t_j$ be the real number such that $(K_X+B+\Mm_X+t_j(D+\NN_X))\cdot C_j=0$. Let $\alpha_j:=(K_X+B+\Mm_X)\cdot C_j$ and $\beta_j:=(K_X+\Delta+\PP_X)\cdot C_j$, then $\alpha_j\in\Ii$, $\beta_j\in\Ii'$, $\alpha_j<0$, $\beta_j\geq 0$, and
$$t_j=\frac{-\alpha_j}{\beta_j-\alpha_j}=\frac{1}{1+\frac{\beta_j}{-\alpha_j}}.$$
Thus $\{t_j\}_{j\in\Lambda}$ is an ACC set, hence
$$t=\sup\{s\geq 0\mid K_X+B+sD+\Mm_X+s\NN_X\text{ is nef}/U\}=\sup_{j\in\Lambda}\{t_j\}=\max_{j\in\Lambda}\{t_j\}=t_{j_0}$$
for some $j_0\in\Lambda$. We may pick $R=R_{j_0}$.
\end{proof}

\section{Proof of Theorems \ref{thm: existence of q-factorial glc flips}, \ref{thm: can run gpair mmp}, and \ref{thm: gpair mmp 3fold and pe fourfold}}

Now we are ready to prove the rest of our main theorems. We start with Theorem \ref{thm: existence of q-factorial glc flips}. In fact, we can prove a slightly stronger result only assuming that $\Mm_X$ is $\Rr$-Cartier. Before we give the proof, let us recall the definitions of flipping contractions and flips.

\begin{defn}[Flipping contraction]\label{defn: flipping contraction}
Let $X\rightarrow U$ be a projective morphism such that $X$ is normal quasi-projective and $D$ an $\Rr$-Cartier $\Rr$-divisor on $X$. A \emph{$D$-flipping contraction} over $U$ is a contraction $f: X\rightarrow Z$ over $U$ satisfying the following:
\begin{enumerate}
    \item $f$ is the contraction of a $D$-negative extremal ray $R$ in $\overline{NE}(X/U)$. In particular, $\rho(X/Z)=1$.
    \item $f$ is small, i.e. $\dim X=\dim Z$ and the exceptional locus of $f$ is of codimension $\geq 2$ in $X$.
\end{enumerate}
\end{defn}

\begin{defn}[Flip]\label{defn: flip}
Let $X$ be a normal quasi-projective variety, $D$ an $\Rr$-Cartier $\Rr$-divisor on $X$, and $f: X\rightarrow Z$ a $D$-flipping contraction. A \emph{$D$-flip} is a birational contraction $f^+: X^+\rightarrow Z$ satisfying the following.
\begin{enumerate}
    \item $D^+$ is $\Rr$-Cartier and ample$/Z$, where $D^+$ is the strict transform of $D$ on $X^+$.
    \item $f^+$ is small.
\end{enumerate}
\end{defn}

\begin{thm}\label{thm: existence glc flip with m r cartier}
Let $(X,B,\Mm)/U$ be an NQC glc g-pair and $f: X\rightarrow Z$ a $(K_X+B+\Mm_X)$-flipping contraction over $U$. Assume that $\Mm_X$ is $\Rr$-Cartier. Then the flip $f^+: X^+\rightarrow Z$ of $f$ exists. 

In particular, $\Mm_{X^+}$ is $\Rr$-Cartier, and if $X$ is $\Qq$-factorial, then $X^+$ is $\Qq$-factorial and $\rho(X)=\rho(X^+)$.
\end{thm}
\begin{proof}
We prove the theorem in three steps. In Step 1, we construct the morphism $f^+: X^+\rightarrow Z$. In Step 2, we show that the morphism $f^+$ constructed in Step 1 is a $(K_X+B+\Mm_X)$-flip. In Step 3, we prove the in particular part of the theorem.

\medskip

\noindent\textbf{Step 1}. In this step,  we construct the morphism $f^+: X^+\rightarrow Z$. 

Let $h: \tilde X\rightarrow X$ be a birational morphism such that $\Mm$ descends to $\tilde X$. Since $\Mm_X$ is $\Rr$-Cartier and $\Mm_{\tilde X}$ is nef$/X$, we have
$$\Mm_{\tilde X}+E=h^*\Mm_X$$
for some $E\geq 0$ that is exceptional over $X$. Let $T\subset X$ be the flipping locus and let $C$ be any flipping curve contracted by $f$. There are two cases:
\medskip

\noindent\textbf{Case 1}. $\Mm_X\cdot C\geq 0$. Then $(K_X+B)\cdot C<0$, and $f$ is also a $(K_X+B)$-flipping contraction. Thus there exists an ample$/Z$ $\Rr$-divisor $A\geq 0$ on $X$ such that $K_X+B+A\sim_{\Rr,Z}0$ and $(X,B+A)$ is lc. By \cite[Theorem 1.1]{Has19}, $(X,B)/U$ has a good minimal model. By Theorem \ref{thm: contraction theorem glc g-pairs}(3), we have $K_X+B\sim_{\Rr,Z}r(K_X+B+\Mm_X)$ for some positive real number $r$. We let $g: Y\rightarrow X$ be a dlt modification of $(X,B)$ and let $K_Y+B_Y=g^*(K_X+B)$, then $K_Y+B_Y+\Mm_Y=g^*(K_X+B+\Mm_X)$, and $(Y,B_Y,\Mm)/U$ and $(Y,B_Y,\mathbf{0})/U$ are glc g-pairs such that $Y$ is $\Qq$-factorial klt. By Lemma \ref{lem: rlin equivalent good minimal model}, $(X,B,\Mm)/Z$ has a good minimal model $(X',B',\Mm)/Z$, and we may let $X'\rightarrow X^+$ be the contraction induced by $K_{X'}+B'+\Mm_{X'}$ over $Z$ and let $f^+: X^+\rightarrow Z$ be the induced morphism.

\medskip

\noindent\textbf{Case 2}. $\Mm_X\cdot C<0$. In this case, $C\subset h(E)$, hence $T\subset h(E)$.  Let $Z^0:= Z\backslash\{f(h(E))\}$, $X^0:=X\times_ZZ^0$, $B^0:=B\times_ZZ^0$, and $\Mm^0:=\Mm\times_ZZ^0$. Since $\Center_XE$ does not contain any glc center of $(X,B,(1-\epsilon)\Mm)$, for any $\epsilon\in (0,1)$,
\begin{itemize}
\item all glc centers of $(X,B,(1-\epsilon)\Mm)$ intersect $X^0$,
\item $(X^0,B^0,(1-\epsilon)\Mm^0)/Z^0$ is a good minimal model of itself (this is because $X^0\cong Z^0$), and
\item $\Mm^0$ descends to $X^0$ and $\Mm^0_{X^0}\sim_{\Rr,Z^0}0$. 
\end{itemize}
Let $\epsilon_0\in (0,1)$ be a real number such that $f$ is also a $(K_X+B+(1-\epsilon_0)\Mm_X)$-flipping contraction. By Theorem \ref{thm: existence of glc closure}, $(X,B,(1-\epsilon_0)\Mm)/Z$ has a good minimal model. Since $\rho(X/Z)=1$, there exists a positive real number $r$ such that $K_X+B+\Mm_X\equiv_{Z}r(K_X+B+(1-\epsilon_0)\Mm_X)$. By Theorem \ref{thm: contraction theorem glc g-pairs}(3), $K_X+B+\Mm_X\sim_{\Rr,Z}r(K_X+B+(1-\epsilon_0)\Mm_X)$. Let $g: Y\rightarrow X$ be a dlt modification of $(X,B)$ and let $K_Y+B_Y:=g^*(K_X+B)$, then $K_Y+B_Y+(1-\epsilon_0)\Mm_Y=g^*(K_X+B+(1-\epsilon_0)\Mm_X)$ and $K_Y+B_Y+\Mm_Y=g^*(K_X+B+\Mm_X)$, and $(Y,B_Y,(1-\epsilon_0)\Mm)/U$ and $(Y,B_Y,\Mm)/U$ are glc g-pairs such that $Y$ is $\Qq$-factorial klt. By Lemma \ref{lem: rlin equivalent good minimal model}, $(X,B,\Mm)/Z$ has a good minimal model $(X',B',\Mm)/Z$, and we may let $X'\rightarrow X^+$ be the contraction induced by $K_{X'}+B'+\Mm_{X'}$ over $Z$ and let $f^+: X^+\rightarrow Z$ be the induced morphism.

\medskip

\noindent\textbf{Step 2}. In this step, we show that the $f^+$ we constructed in Step 1 is a $(K_X+B+\Mm_X)$-flip. Let $B^+$ be the strict transform of $B$ on $X^+$. We only need to check the following two conditions by the definition of a flip:
\begin{enumerate}
\item[(I)] $K_{X^+}+B^++\Mm_{X^+}$ is $\Rr$-Cartier and ample$/Z$.
\item[(II)] $f^+$ is small.
\end{enumerate}
(I) is immediate from our construction. Since $f$ is small, to prove (II), we only need to show that the rational map $X\dashrightarrow X^+$ does not extract any divisor. 

Let $p:W\to X$ and $q:W\to X'$ be a resolution of indeterminacy of $X\dashrightarrow X'$. By Lemma \ref{lem: g-pair version bir12 2.6}, $p^*(K_X+B+\Mm_X)=q^*(K_{X'}+B'+\Mm_{X'})+F$ where $F\geq 0$ is exceptional over $X'$. Let $D$ be a prime divisor on $X'$ that is exceptional over $X$ and $D_W$ its strict transform on $W$. Then $D_W$ is covered by a family of $p$-vertical curves $\Sigma _t$ such that $\Sigma _t\cdot p^*(K_X+B_X+\Mm_X)=0$. Since $F\cdot \Sigma  _t\geq 0$, then $\Sigma _t\cdot q^*(K_{X'}+B'+\Mm_{X'})\leq 0$. 
Let $\Sigma '_t=q_*\Sigma _t$, then $\Sigma ' _t\cdot (K_{X'}+B'+\Mm_{X'})\leq 0$ so that $\Sigma '_t$ are contracted by $X'\to X^+$ and hence $D$ is also contracted. Thus $X\dashrightarrow X^+$ does not extract any divisor, which implies (II). Thus $f^+$ is a $(K_X+B+\Mm_X)$-flip.

\medskip

\noindent\textbf{Step 3}. Now we prove the in particular part of the theorem. Pick any $\Rr$-divisor $D^+$ on $X^+$, and let $D$ be the strict transform of $D^+$ on $X$. 

Assume that $D$ is $\Rr$-Cartier. Since $\rho(X/Z)=1$, there exists a real number $t$ such that $D+t(K_X+B+\Mm_X)\equiv_Z0$. By Theorem \ref{thm: contraction theorem glc g-pairs}(3), $D+t(K_X+B+\Mm_X)\sim_{\Rr,Z}0$. Thus $D+t(K_X+B+\Mm_X)\sim_{\Rr}f^*D_Z$ for some $\Rr$-Cartier $\Rr$-divisor $D_Z$ on $Z$. Therefore, $D^++t(K_{X^+}+B^++\Mm_{X^+})\sim_{\Rr}(f^+)^*D_Z$. Since $K_{X^+}+B^++\Mm_{X^+}$ is $\Rr$-Cartier, $D^+$ is $\Rr$-Cartier.  Therefore, if $\Mm_X$ is $\Rr$-Cartier, then $\Mm_{X^+}$ is $\Rr$-Cartier, and if $X$ is $\Qq$-factorial, then $X^+$ is $\Qq$-factorial.

Since $X\dashrightarrow X^+$ is an isomorphism in codimension $1$, there is a natural isomorphism between the groups of Weil divisors on $X$ and $X^+$. When $X$ and $X^+$ are both $\Qq$-factorial, we have $\rho(X)=\rho(X^+)$, and the proof is concluded.
\end{proof}

\begin{proof}[Proof of Theorem \ref{thm: existence of q-factorial glc flips}] It immediately follows from Theorem  \ref{thm: existence glc flip with m r cartier}.
\end{proof}

\begin{proof}[Proof of Theorem \ref{thm: can run gpair mmp}]
It immediately follows from Theorems \ref{thm: existence glc flip with m r cartier}, \ref{thm: contraction theorem glc g-pairs}, \ref{thm: cone theorem glc g-pairs}, and Corollaries \ref{cor: gpair negative extremal contraction picard number compare} and \ref{cor: gpair divisorial contraction q factoriality}. 
\end{proof}

\begin{proof}[Proof of Theorem \ref{thm: gpair mmp 3fold and pe fourfold}]
It immediately follows from Theorem \ref{thm: can run gpair mmp} and \cite[Corollary 1]{HM20}, \cite[Theorems 1.2,1.3]{CT20}.
\end{proof}

\end{document}